\documentclass{article}

\usepackage[T1]{fontenc}
\usepackage{latexsym}
\usepackage{amsmath}
\usepackage{amssymb}
\usepackage{amsfonts}
\usepackage{amsthm}
\usepackage{amscd}
\usepackage{epsfig}
\usepackage{verbatim}
\usepackage{fancybox}
\usepackage{moreverb}
\usepackage{graphicx}
\usepackage{psfrag}
\usepackage{hyperref}
\usepackage[toc,page]{appendix}
\usepackage[subnum]{cases}
\usepackage{bm}

\newcommand{\R}{{\mathbb R}}
\newcommand{\N}{{\mathbb N}}

\newcommand{\D}{{\partial}}
\newcommand{\Cl}{{\mathcal C}}

\newtheorem{theorem}{Theorem}[section]
\newtheorem{lemma}[theorem]{Lemma}

\newtheorem{definition}[theorem]{Definition}

\newtheorem{hypothesis}[theorem]{Hypothesis}

\def\commutatif{\ar@{}[rd]|{\circlearrowleft}}

\title{Expansion of a singularly perturbed equation with a two-scale converging convection term}

\author{Alexandre MOUTON\footnotemark[1]\thanks{The author wishes to thank E. Fr\'enod, S. Hirstoaga and M. Lutz for the fruitful discussion on the topic and for the ideas which led to some enrichements of the considered models.}}

\date{\today}

\begin{document}
\sloppy

\maketitle

\renewcommand{\thefootnote}{\fnsymbol{footnote}}
\footnotetext[1]{Laboratoire Paul Painlev\'e, CNRS \& Universit\'e de Sciences et Technologies Lille 1, Cit\'e Scientifique, F-59655 Villeneuve d'Ascq, France.
(\texttt{alexandre.mouton@math.univ-lille1.fr})}

\begin{abstract}
In many physical contexts, evolution convection equations may present some very large amplitude convective terms. As an example, in the context of magnetic confinement fusion, the distribution function that describes the plasma satisfies the Vlasov equation in which some terms are of the same order as $\epsilon^{-1}$, $\epsilon \ll 1$ being the characteristic gyrokinetic period of the particles around the magnetic lines. In this paper, we aim to present a model hierarchy for modeling the distribution function for any value of $\epsilon$ by using some two-scale convergence tools. Following Fr\'enod \& Sonnendr\"ucker's recent work, we choose the framework of a singularly perturbed convection equation where the convective terms admit either a high amplitude part or a an oscillating part with high frequency $\epsilon^{-1} \gg 1$. In this abstract framework, we derive an expansion with respect to the small parameter $\epsilon$ and we recursively identify each term of this expansion. Finally, we apply this new model hierarchy to the context of a linear Vlasov equation in three physical contexts linked to the magnetic confinement fusion and the evolution of charged particle beams.
\end{abstract}

\paragraph{Key words.}
Vlasov equation, Two-scale convergence, Gyrokinetic approximations, convection equation.

\paragraph{AMS subject classifications.}
35Q83, 76M40, 78A35, 82D10.

\pagestyle{myheadings}
\thispagestyle{plain}
\markboth{A. MOUTON}{TWO-SCALE EXPANSION OF A CONVECTION EQUATION}

\section{Introduction}

For sixty years, Magnetic Confinement Fusion (MCF) is one of the most important technological challenges for producing domestic energy. Indeed, this worldwide project involves physicists, engineers and mathematicians in order to understand and reproduce on Earth the solar magnetic fusion reaction. One of the most famous examples of this work programme is the ITER project localized in Cadarache (France) which attempts to produce a fusion plasma in a tokamak reactor by confining it thanks to a strong external magnetic field. Besides the required technological aspects of MCF, it became necessary for thirty years to lead a rigorous study of the behaviour of such a plasma and this work takes the form of the derivation of mathematical models and of high precision numerical experiments. \\
\indent In the present paper, we focus on the Vlasov equation in presence of a external magnetic field with an amplitude of the same order as $\epsilon^{-1} \gg 1$ and on its limit regime as $\epsilon \to 0$. Such an equation is the main subject of many previous works: indeed, many results about the mathematical justifications of Guiding-Center and Finite Larmor Radius limit regimes have been established by Bostan \cite{Bostan_2007,Bostan_2010}, Fr\'enod \& Sonnendr\"ucker \cite{Two-scale_expansion,Homogenization_VP,Finite_Larmor_radius}, Fr\'enod \& Mouton \cite{Frenod-Mouton,PhD_Mouton}, Golse \& Saint-Raymond \cite{Golse_1,Golse_2} and Han-Kwan \cite{Han-Kwan_3,Han-Kwan_2,Han-Kwan}. Most of these results are based on the use of two-scale convergence and homogenization techniques (see Allaire \cite{Allaire} and Nguetseng \cite{NGuetseng}) or compactness methods. These mathematical studies allowed to validate and reinforce the tokamak plasma models presented by Littlejohn, Lee \textit{et al.}, Dubin \textit{et al.} or Brizard \textit{et al.} (see \cite{Littlejohn}, \cite{Lee, Lee_2}, \cite{Dubin}, \cite{Brizard_PhD, Hahm-Brizard}). \\
\indent The linear Vlasov equations we are focused on in the present paper are the following:
\begin{equation} \label{Vlasov_GCeps_intro}
\left\{
\begin{array}{l}
\displaystyle \D_{t}f_{\epsilon}+\mathbf{v}\cdot\nabla_{\mathbf{x}}f_{\epsilon} + \left(\mathbf{E}_{\epsilon} + \mathbf{v} \times \mathbf{B}_{\epsilon} + \cfrac{\mathbf{v} \times \bm{\beta}_{\epsilon}}{\epsilon} \right) \cdot \nabla_{\mathbf{v}}f_{\epsilon} = 0 \, , \\
f_{\epsilon}(t=0,\mathbf{x},\mathbf{v}) = f^{0}(\mathbf{x},\mathbf{v}) \, ,
\end{array}
\right.
\end{equation}
\begin{equation}\label{Vlasov_FLReps_intro}
\left\{
\begin{array}{l}
\displaystyle \D_{t}f_{\epsilon} + \cfrac{\mathbf{v}_{\perp}}{\epsilon}\cdot\nabla_{\mathbf{x}_{\perp}}f_{\epsilon} +v_{||}\,\D_{x_{||}}f_{\epsilon}+ \left(\mathbf{E}_{\epsilon} + \mathbf{v} \times \mathbf{B}_{\epsilon} + \cfrac{\mathbf{v} \times \bm{\mathcal{M}}}{\epsilon} \right) \cdot \nabla_{\mathbf{v}}f_{\epsilon} = 0 \, , \\
f_{\epsilon}(t=0,\mathbf{x},\mathbf{v}) = f^{0}(\mathbf{x},\mathbf{v}) \, ,
\end{array}
\right.
\end{equation}
\begin{equation}\label{Vlasov_axibeam_intro}
\left\{
\begin{array}{l}
\displaystyle \D_{t}f_{\epsilon}(t,r,v_{r}) + \cfrac{v_{r}}{\epsilon}\,\D_{r}f_{\epsilon}(t,r,v_{r}) + \left(E_{\epsilon}(t,r)-\cfrac{r}{\epsilon}\right)\,\D_{v_{r}}f_{\epsilon}(t,r,v_{r}) = 0 \, , \\
f_{\epsilon}(t=0,r,v_{r}) = f^{0}(r,v_{r}) \, .
\end{array}
\right.
\end{equation}
\indent The two first equations can be seen in kinetic models for magnetic fusion plasma. In such context, $f_{\epsilon} = f_{\epsilon}(t,\mathbf{x},\mathbf{v})$ is the distribution function that describes the evolution of the plasma in the phase space, $\mathbf{E}_{\epsilon} = \mathbf{E}_{\epsilon}(t,\mathbf{x})$ and $\mathbf{B}_{\epsilon}=\mathbf{B}_{\epsilon}(t,\mathbf{x})$ are the external electric and magnetic fields that are applied on the plasma, $\bm{\beta}_{\epsilon} = \bm{\beta}_{\epsilon}(t,\mathbf{x})$ is a given vector function assumed to oscillate in time with $\mathcal{O}(\epsilon^{-1})$ order frequency, $\bm{\mathcal{M}}$ is a fixed unit vector in $\R^{3}$ allowing to define, for any $\mathbf{v} \in \R^{3}$, $v_{||} = \bm{\mathcal{M}} \cdot \mathbf{v}$ and $\mathbf{v}_{\perp} = \mathbf{v}-v_{||}\,\bm{\mathcal{M}}$. Finally, $t$, $\mathbf{x}$ and $\mathbf{v}$ stand for the time, position and velocity variables. Both equations \eqref{Vlasov_GCeps_intro} and \eqref{Vlasov_FLReps_intro} can be derived from the collisionless Vlasov equation
\begin{equation*}\label{Vlasov}
\left\{
\begin{array}{l}
\displaystyle \D_{t}f + \mathbf{v} \cdot \nabla_{\mathbf{x}}f + \frac{q}{m}\left(\mathbf{E}+\mathbf{v}\times \mathbf{B}\right) \cdot \nabla_{\mathbf{v}} f = 0 \, , \\
f(t=0,\mathbf{x},\mathbf{v}) = f^{0}(\mathbf{x},\mathbf{v}) \, ,
\end{array}
\right.
\end{equation*}
where a rescaling is considered: $\epsilon$ stands for the ratio between the characteristic gyro-period of the particles and the characteristic duration of the experiment and the ratio between the electric force amplitude and the magnetic force amplitude. In order to differenciate the derivation of \eqref{Vlasov_GCeps_intro} from \eqref{Vlasov_FLReps_intro}, we assume that the characteristic Larmor radius is of the same order as $\epsilon$ in front of any characteristic length to obtain \eqref{Vlasov_GCeps_intro}, whereas we assume that it is of the same order as $\epsilon$ in front of the characteristic length in $\bm{\mathcal{M}}$-direction and of the same order as the characteristic length in the orthogonal plane to $\bm{\mathcal{M}}$ for obtaining \eqref{Vlasov_FLReps_intro}. The details of such derivations can be found in \cite{Homogenization_VP} for \eqref{Vlasov_GCeps_intro} and in \cite{Finite_Larmor_radius} for \eqref{Vlasov_FLReps_intro}. \\
\indent Equation \eqref{Vlasov_axibeam_intro} can be encountered in the context of axisymmetric charged particle beam submitted to a external electric that oscillates  with high frequency. In this context, $f_{\epsilon} = f_{\epsilon}(t,r,v_{r})$ is the distribution function of the particles that are submitted to the focusing electric field $E_{\epsilon}(t,r)-\frac{r}{\epsilon}$, $t$, $r$ and $v_{r}$ stand for the pseudo-time, radial position and radial velocity variables. Such equation can be derived from the paraxial approximation of the Vlasov equation given by
\begin{equation*}\label{Vlasov_paraxial}
\left\{
\begin{array}{l}
\D_{z}f + \cfrac{\mathbf{v}}{v_{z}}\cdot \nabla_{\mathbf{x}}f + \cfrac{q}{\gamma_{z}m v_{z}}\, \left( \cfrac{\mathbf{E}}{\gamma_{z}^{2}} - H_{0}\mathbf{x} \right) \cdot \nabla_{\mathbf{v}}f = 0 \, , \\
f(\mathbf{x},z=0,\mathbf{v}) = f^{0}(\mathbf{x},\mathbf{v}) \, ,
\end{array}
\right.
\end{equation*}
where $(\mathbf{x},z) \in \R^{2}\times \R_{+}$ is the position variable, $\mathbf{v} \in \R^{2}$ is the velocity variable in the perpendicular plane to $z$-direction, $f = f(\mathbf{x},z,\mathbf{v})$ is the distribution function of the particles with constant longitudinal velocity $v_{z}$, $\gamma_{z}$ is the time dilatation coefficient associated to $v_{z}$, $\mathbf{E}$ and $\bm{\Xi}$ are respectively the self-consistent and external electric fields, and $H_{0}$ is a positive constant tension. Such Vlasov equation can be derived from the stationary Vlasov-Maxwell system (see \cite{Degond-Raviart,Filbet-Sonnen}). To obtain \eqref{Vlasov_axibeam_intro}, we assume that the beam is long and thin so the ratio $\epsilon$ between the characteristic radius of the beam and its characteristic length in the propagation direction is small, we assume that the angular momentum is equal to zero at the beam source $z=0$, and we consider polar coordinates in $\mathbf{x}$ and $\mathbf{v}$. Details on such derivation can be found in \cite{PIC-two-scale, Mouton_2009}. \\

\indent The main goal of the present paper is to study the two-scale asymptotic behaviour of the distribution function $f_{\epsilon}$ when $\epsilon$ converges to 0 for each Vlasov equation \eqref{Vlasov_GCeps_intro}, \eqref{Vlasov_FLReps_intro} and \eqref{Vlasov_axibeam_intro}. Some papers already provide two-scale convergence results of each of these models. Indeed, in \cite{Homogenization_VP}, Fr\'enod \& Sonnendr\"ucker studied the two-scale convergence of the solution $f_{\epsilon}$ of \eqref{Vlasov_GCeps_intro} as $\epsilon \to 0$: they proved that the sequence $(f_{\epsilon})_{\epsilon\,>\,0}$ admits a 0-th order two-scale limit $F_{0}$ when $\epsilon$ converges to 0 in the case where $\mathbf{B}_{\epsilon} = 0$ and $\bm{\beta}_{\epsilon}$ is a $\epsilon$-independent uniform vector in space and time. In \cite{Finite_Larmor_radius}, they establish a similar 0-th order two-scale convergence result for the solution $f_{\epsilon}$ of the model \eqref{Vlasov_FLReps_intro} in the 4D+time case where the model does not depend on $x_{||}$ nor $v_{||}$ and where $\mathbf{B}_{\epsilon} = 0$. Furthermore, in \cite{Two-scale_expansion}, the authors establish a $k$-th order two-scale convergence result for the solution $f_{\epsilon}$ of the 6D+time equation \eqref{Vlasov_FLReps_intro} with $k \in \N$ arbitrarily chosen: in this paper, the external electric field $\mathbf{E}_{\epsilon}$ is assumed to be independent of $\epsilon$ and $\mathbf{B}_{\epsilon} = 0$. The authors prove that $f_{\epsilon}$ two-scale converges at $k$-order to a profile $F_{k}$ thanks to a recursive procedure on a generic singularly perturbed convection equation. Some two-scale convergence results have also been established for the solution of \eqref{Vlasov_axibeam_intro}. Indeed, in \cite{PIC-two-scale}, the authors established a 0-th order two-scale convergence by proving that $f_{\epsilon}$ two-scale converges to a profile $F_{0}$ as $\epsilon$ tends to 0. In \cite{Frenod-Gutnic-Hirstoaga}, this result is extended to the first order. Indeed, introducing $f_{\epsilon,1}$ defined as
\begin{equation*}
f_{\epsilon,1}(t,r,v_{r}) = \cfrac{1}{\epsilon}\,\left( f_{\epsilon}(t,r,v_{r})-F_{0}\left(t,\cfrac{t}{\epsilon},r,v_{r}\right)\right) \, ,
\end{equation*}
the authors proved that the sequence $(f_{\epsilon,1})_{\epsilon\,>\,0}$ two-scale converges to a profile $F_{1}$ and provide a limit system satisfied by $F_{1}$ by assuming that $E_{\epsilon}(t,r) = E_{0}(t,\frac{t}{\epsilon},r) + \epsilon\,E_{1}(t,\frac{t}{\epsilon},r)$ with $\epsilon$-independent functions $E_{0}$ and $E_{1}$. \\
\indent The aim of the present document is to generalize the two-scale convergence results on \eqref{Vlasov_GCeps_intro}-\eqref{Vlasov_FLReps_intro}-\eqref{Vlasov_axibeam_intro} presented in \cite{Frenod-Gutnic-Hirstoaga, Two-scale_expansion, PIC-two-scale, Homogenization_VP, Finite_Larmor_radius}. More precisely, we aim to generalize the two-scale convergence results on \eqref{Vlasov_GCeps_intro} to the $k$-order and with an non-zero $\mathbf{B}_{\epsilon}$ and a varying $\bm{\beta}_{\epsilon}$. Our goal is also to generalize the results on \eqref{Vlasov_FLReps_intro} established in \cite{Two-scale_expansion} to the case with non-zero $\epsilon$-dependent external fields $\mathbf{E}_{\epsilon}$ and $\mathbf{B}_{\epsilon}$. Finally, we aim to extend the results from \cite{Frenod-Gutnic-Hirstoaga} to the k-order of two-scale convergence. For this, we consider the following generic singular perturbed convection equation that includes the linear Vlasov equations \eqref{Vlasov_GCeps_intro}, \eqref{Vlasov_FLReps_intro} and \eqref{Vlasov_axibeam_intro}:
\begin{equation}\label{convection}
\left\{
\begin{array}{l}
\displaystyle \D_{t}u_{\epsilon}(t,\mathbf{x}) + \mathbf{A}_{\epsilon}(t,\mathbf{x})\cdot \nabla_{\mathbf{x}}u_{\epsilon}(t,\mathbf{x}) + \cfrac{1}{\epsilon} \, \mathbf{L}\left(t,\cfrac{t}{\epsilon},\mathbf{x}\right)\cdot \nabla_{\mathbf{x}}u_{\epsilon}(t,\mathbf{x}) = 0 \, , \\
u_{\epsilon}(t=0,\mathbf{x}) = u^{0}(\mathbf{x}) \, ,
\end{array}
\right.
\end{equation}
where $t \in [0,T]$ and $\mathbf{x} \in \R^{n}$ ($n \in \N^{*}$) are the variables ($T > 0$ is fixed), $\mathbf{A}_{\epsilon}:[0,T] \times \R^{n} \to \R^{n}$ and $\mathbf{L}:[0,T]\times\R\times\R^{n} \to \R^{n}$ are given vector functions and the solution quantity is $u_{\epsilon}:[0,T]\times\R^{n} \to \R$. We fix $\theta > 0$ and we assume that $\mathbf{A}_{\epsilon}$ and $\mathbf{L}$ are divergence-free in $\mathbf{x}$-direction and that $\mathbf{L}$ is $\theta$-periodic in $\tau$-direction, \textit{i.e.}
\begin{align*}
\forall\,(t,\tau,\mathbf{x}) \in [0,T]\times\R\times\R^{n} \, , \qquad &\nabla_{\mathbf{x}} \cdot \mathbf{A}_{\epsilon}(t,\mathbf{x}) = \nabla_{\mathbf{x}} \cdot \mathbf{L}(t,\tau,\mathbf{x}) = 0 \, , \\
\forall\,(t,\tau,\mathbf{x}) \in [0,T]\times\R\times\R^{n} \, , \qquad &\mathbf{L}(t,\tau+\theta,\mathbf{x}) = \mathbf{L}(t,\tau,\mathbf{x}) \, .
\end{align*}
We also assume that, for any fixed $\epsilon > 0$, the initial data $u^{0}$ and the vector functions $\mathbf{A}_{\epsilon}$ and $\mathbf{L}$ satisfy the minimal required smoothness properties for insuring the existence and the uniqueness of the solution $u_{\epsilon}$ of \eqref{convection}. This generic model is close to the convection equation studied in \cite{Two-scale_expansion}. Indeed, in this paper, the convection term $\mathbf{A}_{\epsilon}$ does not depend on $\epsilon$ and $\mathbf{L}$ only depends on $t$ and $\mathbf{x}$. \\

\indent Thus, the present paper is organized as follows: in Section 2, we present a two-scale convergence theorem for the generic convection equation \eqref{convection} then we use it to extend in a straightforward way the existing two-scale convergence results for the solution of each Vlasov equation \eqref{Vlasov_GCeps_intro}, \eqref{Vlasov_FLReps_intro} and \eqref{Vlasov_axibeam_intro}. In the following Section, we describe the proof for obtaining the two-scale convergence theorem on \eqref{convection}. In a last section, we will discuss some perspectives for future work.

\section{Two-scale convergence results}

In this section, we present the main results of the present paper. After recalling some definitions and notations that will be used along the paper, we first present a 0-th order two-scale convergence result for the solution $u_{\epsilon}$ of the generic convection equation \eqref{convection}. Secondly we detail the required hypotheses for reaching the $k$-order two-scale convergence of $u_{\epsilon}$, then the result itself. Finally, we adapt these results for \eqref{convection} to each linear Vlasov equation \eqref{Vlasov_GCeps_intro}, \eqref{Vlasov_FLReps_intro} and \eqref{Vlasov_axibeam_intro}.

\subsection{Notations and definitions}

\indent Before going further and presenting the main results, we introduce some notations and definitions. Considering a fixed $\theta > 0$, we define for any $p \in[1,+\infty]$ the space $L_{\#}^{p}(0,\theta)$ as the functions $f : \R \to \R$ that are $\theta$-periodic and such that $f_{|_{[0,\theta]}} \in L^{p}(0,\theta)$. In the same spirit, we define $\Cl_{\#}(0,\theta)$ stands for the subspace of $\Cl(\R)$ constituted of $\theta$-periodic functions and provided with the the norm induced by $\Cl(\R)$. Having these notations in hands, we recall the definition of two-scale convergence as it has been introduced by Allaire \cite{Allaire} and Nguetseng \cite{NGuetseng} and a useful two-scale convergence criterion: \\

\begin{definition}
Let $X$ be a separable Banach space, $X'$ its topological dual space, and $\langle \cdot , \cdot \rangle_{X,X'}$ the duality bracket associated to $X$ and $X'$. Considering fixed $q \in [1,+\infty[$, $T > 0$, and $q'$ such that $\frac{1}{q}+\frac{1}{q'} = 1$, a sequence $(u_{\epsilon})_{\epsilon\,>\,0} \subset L^{q'}(0,T;X')$ two-scale converges to a function $U \in L^{q'}\left(0,T;L_{\#}^{q'}(0,\theta;X')\right)$ if, for any test function $\psi \in L^{q}\left(0,T;\Cl_{\#}\left(0,\theta;X\right)\right)$, we have
\begin{equation*}
\lim_{\epsilon\,\to\,0} \int_{0}^{T} \left\langle u_{\epsilon}(t), \psi\left(t,\cfrac{t}{\epsilon}\right)\right\rangle_{X,X'} \, dt = \int_{0}^{T} \int_{0}^{\theta} \left\langle U(t,\tau), \psi\left(t,\tau\right)\right\rangle_{X,X'} \, d\tau\,dt \, .
\end{equation*}
\end{definition}

\begin{theorem}[Allaire \cite{Allaire}]\label{TSCV_Allaire}
If a sequence $(u_{\epsilon})_{\epsilon\,>\,0} \subset L^{q'}(0,T;X')$ is bounded independently of $\epsilon$, there exists a profile $U \in L^{q'}\left(0,T;L_{\#}^{q'}(0,\theta;X')\right)$ such that, up to the extraction of a subsequence
\begin{equation*}
u_{\epsilon} \, \longrightarrow \, U \qquad \textit{two-scale in $L^{q'}\left(0,T;L_{\#}^{q'}(0,\theta;X')\right)$.}
\end{equation*}
Furthermore, the so-called two-scale limit $U$ of $u_{\epsilon}$ is closely linked to the weak-* limit of $(u_{\epsilon})_{\epsilon\,>\,0}$ in $L^{q'}(0,T;X')$. Indeed this function denoted by $u$ satisfies
\begin{equation*}
u(t) = \cfrac{1}{\theta} \, \int_{0}^{\theta} U(t,\tau) \, d\tau \, .
\end{equation*}
\end{theorem}

\subsection{The singularly perturbed convection equation}

For any $(t,\sigma,\mathbf{x}) \in [0,T]\times\R\times\R^{n}$ fixed, we consider the following differential system
\begin{equation*} 
\left\{
\begin{array}{rcl}
\D_{\tau}\mathbf{X}(\tau) &=& \mathbf{L}\left(t,\tau,\mathbf{X}(\tau)\right) \, , \\
\mathbf{X}(\sigma) &=& \mathbf{x} \, ,
\end{array}
\right.
\end{equation*}
where the unknown is the vector function $\tau \mapsto \mathbf{X}(\tau)$. We assume from now that this system admits a unique solution in the class of $\theta$-periodic functions in $\tau$-direction and we denote this solution by $\tau \mapsto \mathbf{X}(\tau;\mathbf{x},t;\sigma)$. \\
\indent In the following lines, we aim to write the following development of $u_{\epsilon}$
\begin{equation} \label{expansion}
u_{\epsilon}(t,\mathbf{x}) = \sum_{k\,=\,0}^{+\infty} \epsilon^{k}\,U_{k}\left(t,\cfrac{t}{\epsilon},\mathbf{x}\right) \, ,
\end{equation}
and to characterize successively the terms $U_{k}$ of this expansion.

\subsubsection{0-th order convergence} \label{TSCV_U0}
The first main result is the two-scale convergence of $(u_{\epsilon})_{\epsilon\,>\,0}$ to a profile $U_{0} = U_{0}(t,\tau,\mathbf{x})$. For this purpose, we consider some hypotheses derived from those which are required for proving Theorem 1.5 of \cite{Finite_Larmor_radius}:

\begin{hypothesis}\label{hyp_U0}
Fixing $p \in \, ]1,+\infty[$, $q > 1$ and $q'$ such that $\frac{1}{p}+\frac{1}{q'} < 1$ and $\frac{1}{q'} = \max(\frac{1}{q}-\frac{1}{n},0)$, we assume that
\begin{itemize}
\item $u^{0} \in L^{p}(\R^{n})$,
\item $(\mathbf{A}_{\epsilon})_{\epsilon\,>\,0}$ is bounded independently of $\epsilon$ in $\left(L^{\infty}\left(0,T;\left(W^{1,q}(K)\right) \right)\right)^{n}$ for any compact subset $K \subset \R^{n}$,
\item $\mathbf{L}$ is smooth enough in order to insure that, for any compact subset $K \subset \R^{n}$,
\begin{itemize}
\item $\mathbf{L}$ is in $\left(L^{\infty}\left(0,T;L_{\#}^{\infty}\left(0,\theta;W^{1,q}(K)\right)\right)\right)^{n}$,
\item $(t,\tau,\mathbf{x}) \mapsto \D_{t}\mathbf{X}(\tau;\mathbf{x},t;0)$ is in $\left(L^{\infty}\left(0,T;L_{\#}^{\infty}\left(0,\theta;W^{1,q}(K)\right)\right)\right)^{n}$,
\item $(t,\tau,\mathbf{x}) \mapsto \nabla_{\mathbf{x}}\mathbf{X}(\tau;\mathbf{x},t;0)$ is in $\left(L^{\infty}\left(0,T;L_{\#}^{\infty}\left(0,\theta;L^{\infty}(K)\right)\right)\right)^{n^{2}}$.
\end{itemize}
\end{itemize}
\end{hypothesis}

As a trivial consequence, we can write, up to a subsequence and for any compact $K \subset \R^{n}$,
\begin{equation*}
\mathbf{A}_{\epsilon}\, \longrightarrow \, \bm{\mathcal{A}}_{0} = \bm{\mathcal{A}}_{0}(t,\tau,\mathbf{x}) \qquad \textnormal{two-scale in $\left(L^{\infty}\left(0,T;L_{\#}^{\infty}\left(0,\theta;\left(W^{1,q}(K)\right) \right)\right)\right)^{n}$.}
\end{equation*}
Assuming that the profile $\bm{\mathcal{A}}_{0}$ is somehow known, we introduce $\bm{\alpha}_{0}$ and $\tilde{\mathbf{a}}_{0}$ as
\begin{equation*} 
\bm{\alpha}_{0}(t,\tau,\mathbf{x}) = \left((\nabla_{\mathbf{x}}\mathbf{X})(\tau;\mathbf{x},t;0)\right)^{-1} \,\left( \bm{\mathcal{A}}_{0}\left(t,\tau,\mathbf{X}(\tau;\mathbf{x},t;0)\right) - (\D_{t}\mathbf{X})(\tau;\mathbf{x},t;0)\right) \, ,
\end{equation*}
and
\begin{equation*} 
\tilde{\mathbf{a}}_{0}(t,\mathbf{x}) = \cfrac{1}{\theta}\int_{0}^{\theta} \bm{\alpha}_{0}(t,\tau,\mathbf{x}) \, d\tau \, .
\end{equation*}

With these hypotheses and definitions, we can characterize the 0-th order term $U_{0}$ in the expansion \eqref{expansion}:

\begin{theorem}\label{def_U0}
Assume that Hypotheses \ref{hyp_U0} and that the sequence $(u_{\epsilon})_{\epsilon\,>\,0}$ is bounded in $L^{\infty}\left(0,T;L_{loc}^{p}(\R^{n})\right)$ independently of $\epsilon$. Then, up to a subsequence, $u_{\epsilon}$ two-scale converges to the profile $U_{0} = U_{0}(t,\tau,\mathbf{x})$ in $L^{\infty}\left(0,T;L_{\#}^{\infty}\left(0,\theta;L^{p}(\R^{n})\right)\right)$ defined by
\begin{equation}\label{link_U0V0}
U_{0}(t,\tau,\mathbf{x}) = V_{0}\left(t,\mathbf{X}(-\tau;\mathbf{x},t;0)\right) \, ,
\end{equation}
where $V_{0} = V_{0}(t,\mathbf{x}) \in L^{\infty}\left(0,T;L_{loc}^{p}(\R^{n})\right)$ satisfies
\begin{equation}\label{eq_V0}
\left\{
\begin{array}{l}
\displaystyle \D_{t}V_{0}(t,\mathbf{x}) + \tilde{\mathbf{a}}_{0}(t,\mathbf{x}) \cdot \nabla_{\mathbf{x}}V_{0}(t,\mathbf{x}) = 0 \, , \\
V_{0}(t=0,\mathbf{x}) = u^{0}(\mathbf{x}) \, .
\end{array}
\right.
\end{equation}
\end{theorem}

\begin{theorem} \label{eq_U0}
$U_{0}$ satisfies the following equation:
\begin{equation*}
\D_{t}U_{0}(t,\tau,\mathbf{x}) + \mathbf{a}_{0}(t,\tau,\mathbf{x}) \cdot \nabla_{\mathbf{x}}U_{0}(t,\tau,\mathbf{x}) = 0 \, ,
\end{equation*}
with $\mathbf{a}_{0}$ defined by
\begin{equation*}
\begin{split}
\mathbf{a}_{0}(t,\tau,\mathbf{x}) = \left((\nabla_{\mathbf{x}}\mathbf{X})(-\tau;\mathbf{x},t;0)\right)^{-1}\left(\tilde{\mathbf{a}}_{0}\left(t,\mathbf{X}(-\tau;\mathbf{x},t;0)\right) - (\D_{t}\mathbf{X})(-\tau;\mathbf{x},t;0) \right) \, .
\end{split}
\end{equation*}
\end{theorem}

\subsubsection{Two-scale convergence at $k$-th order} \label{TSCV_Uk}

We fix $k \in \N^{*}$ and we aim to identify the $k$-th term of the expansion \eqref{expansion}. Before stating the result, we need additional assumptions besides Hypotheses \ref{hyp_U0}: 

\begin{hypothesis} \label{Hilbert_Aeps}
Defining the sequence $(\mathbf{A}_{\epsilon,i})_{\epsilon\,>\,0}$ as
\begin{equation*}
\left\{
\begin{array}{ll}
\mathbf{A}_{\epsilon,i}(t,\mathbf{x}) = \cfrac{1}{\epsilon}\,\left(\mathbf{A}_{\epsilon,i-1}(t,\mathbf{x}) - \bm{\mathcal{A}}_{i-1}\left(t,\cfrac{t}{\epsilon},\mathbf{x}\right) \right) \, , & \forall\,i=1,\dots,k \, , \\
\mathbf{A}_{\epsilon,0}(t,\mathbf{x}) = \mathbf{A}_{\epsilon}(t,\mathbf{x}) \, ,
\end{array}
\right.
\end{equation*}
we assume that, for all $i=0,\dots,k$, $(\mathbf{A}_{\epsilon,i})_{\epsilon\,>\,0}$ two-scale converges to the profile $\bm{\mathcal{A}}_{i} = \bm{\mathcal{A}}_{i}(t,\tau,\mathbf{x})$ in $\left(L^{\infty}\left(0,T;L_{\#}^{\infty}\left(0,\theta;W^{1,q}(K)\right)\right)\right)^{n}$ for any compact subset $K \subset \R^{n}$.
\end{hypothesis} 

\begin{hypothesis} \label{Hilbert_ueps_kth}
Defining the sequence $(u_{\epsilon,i})_{\epsilon\,>\,0}$ as
\begin{equation*} \label{def_uepsi}
\left\{
\begin{array}{ll}
u_{\epsilon,i}(t,\mathbf{x}) = \cfrac{1}{\epsilon}\,\left(u_{\epsilon,i-1}(t,\mathbf{x}) - U_{i-1}\left(t,\cfrac{t}{\epsilon},\mathbf{x}\right) \right) \, , & \forall\, i=1,\dots,k-1 \, , \\
u_{\epsilon,0}(t,\mathbf{x}) = u_{\epsilon}(t,\mathbf{x}) \, ,
\end{array}
\right.
\end{equation*}
we assume that, for all $i=0,\dots,k-1$ and up to a subsequence, $(u_{\epsilon,i})_{\epsilon\,>\,0}$ two-scale converges to the profile $U_{i} = U_{i}(t,\tau,\mathbf{x})$ in $L^{\infty}\left(0,T;L_{\#}^{\infty}\left(0,\theta;L^{p}(\R^{n})\right)\right)$.
\end{hypothesis}

Under these hypotheses, we define $\bm{\alpha}_{i}$, $\tilde{\mathbf{a}}_{i}$ and $\mathbf{a}_{i}$ as
\begin{equation*}
\bm{\alpha}_{i}(t,\tau,\mathbf{x}) = \left((\nabla_{\mathbf{x}}\mathbf{X})(\tau;\mathbf{x},t;0)\right)^{-1}\,\bm{\mathcal{A}}_{i}\left(t,\tau,\mathbf{X}(\tau;\mathbf{x},t;0)\right) \, ,
\end{equation*}
\begin{equation*}
\tilde{\mathbf{a}}_{i}(t,\mathbf{x}) = \cfrac{1}{\theta}\, \int_{0}^{\theta}\bm{\alpha}_{i}(t,\tau,\mathbf{x}) \, d\tau \, ,
\end{equation*}
\begin{equation*}
\begin{split}
\mathbf{a}_{i}(t,\tau,\mathbf{x}) &= \cfrac{1}{\theta}\,\int_{0}^{\theta} \left((\nabla_{\mathbf{x}}\mathbf{X})(\sigma-\tau;\mathbf{x},t;0)\right)^{-1}\,\bm{\mathcal{A}}_{i}\left(t,\sigma,\mathbf{X}(\sigma-\tau;\mathbf{x},t;0)\right) \, d\sigma \\
&= \cfrac{1}{\theta}\, \int_{0}^{\theta} \left((\nabla_{\mathbf{x}}\mathbf{X})(-\tau;\mathbf{x},t;0)\right)^{-1}\,\bm{\alpha}_{i}\left(t,\sigma,\mathbf{X}(-\tau;\mathbf{x},t;0)\right) \, d\sigma \, ,
\end{split}
\end{equation*}
for all $i=1,\dots,k-1$. We also define recursively the functions $W_{1},\dots,W_{k}$ and $R_{1},\dots,R_{k-1}$ as
\begin{equation} \label{def_Wi}
W_{i}(t,\tau,\mathbf{x}) = \int_{0}^{\tau} \left(\sum_{j\,=\,0}^{i-1}(\mathbf{a}_{j}-\bm{\mathcal{A}}_{j})\cdot\nabla_{\mathbf{x}}U_{i-1-j} - R_{i-1}\right)\left(t,\sigma,\mathbf{X}(\sigma;\mathbf{x},t;0)\right) \, d\sigma \, ,
\end{equation}
\begin{equation} \label{def_Ri}
\begin{split}
R_{i}(t,\tau,\mathbf{x}) 
&= (\D_{t}W_{i})\left(t,\tau,\mathbf{X}(-\tau;\mathbf{x},t;0)\right) -\cfrac{1}{\theta}\int_{0}^{\theta}(\D_{t}W_{i})\left(t,\sigma,\mathbf{X}(-\tau;\mathbf{x},t;0)\right)\, d\sigma \\
&\qquad + \sum_{j\,=\,0}^{i}\left(\tilde{\mathbf{a}}_{j}\cdot\nabla_{\mathbf{x}}W_{i-j}\right)\left(t,\tau,\mathbf{X}(-\tau;\mathbf{x},t;0)\right) \\
&\qquad - \cfrac{1}{\theta}\sum_{j\,=\,0}^{i} \int_{0}^{\theta} \left(\bm{\alpha}_{j}\cdot\nabla_{\mathbf{x}}W_{i-j}\right) \left(t,\sigma,\mathbf{X}(-\tau;\mathbf{x},t;0)\right)\,d\sigma \, ,
\end{split}
\end{equation}
with the convention $W_{0} = R_{0} = 0$. Having these notations and assumptions in hands, we can now state a two-scale convergence result at $k$-th order by proceeding recursively:

\begin{theorem} \label{CV_Uk}
We define $s' > 0$ such that $\frac{1}{s'} = 1-\frac{1}{q}-\frac{1}{r}$ with $r \in [1,\frac{nq}{n-q}[$ and we define the functional space $X^{s'}(K) = \left(W^{1,q}(K)\right)' \cup \left(W^{1,s'}(K)\right)'$. We assume that Hypotheses \ref{hyp_U0}-\ref{Hilbert_Aeps}-\ref{Hilbert_ueps_kth} are satisfied and that, for any $K \subset \R^{n}$ compact, 
\begin{itemize}
\item $W_{k}$ is in $L^{\infty}\left(0,T;L_{\#}^{\infty}\left(0,\theta;L^{p}(K)\right)\right)$,
\item $\D_{t}W_{k}$ is in $L^{\infty}\left(0,T;L_{\#}^{\infty}\left(0,\theta;X^{s'}(K)\right)\right)$,
\item $R_{k-1}$ is in $L^{\infty}\left(0,T;L_{\#}^{\infty}\left(0,\theta;X^{s'}(K)\right)\right)$.
\end{itemize}
Then, if the sequence $(u_{\epsilon,k})_{\epsilon\,>\,0}$ defined by 
\begin{equation*}
u_{\epsilon,k}(t,\mathbf{x}) = \cfrac{1}{\epsilon}\,\left(u_{\epsilon,k-1}(t,\mathbf{x})-U_{k-1}\left(t,\cfrac{t}{\epsilon},\mathbf{x}\right)\right)\,,
\end{equation*}
is bounded independently of $\epsilon$ in $L^{\infty}\left(0,T;L_{loc}^{p}(\R^{n})\right)$, $u_{\epsilon,k}$ two-scale converges to the profile $U_{k}$ in $\left(0,T;L_{\#}^{\infty}\left(0,\theta;L^{p}(\R^{n})\right)\right)$ characterized as follows:
\begin{equation}
U_{k}(t,\tau,\mathbf{x}) = V_{k}\left(t,\mathbf{X}(-\tau;\mathbf{x},t;0)\right) + W_{k}\left(t,\tau,\mathbf{X}(-\tau;\mathbf{x},t;0)\right) \, ,
\end{equation}
where $W_{k} = W_{k}(t,\tau,\mathbf{x})$ is defined in \eqref{def_Wi} and where $V_{k} = V_{k}(t,\mathbf{x})$ satisfies
\begin{equation} \label{eq_Vk}
\left\{
\begin{array}{l}
\begin{split}
\D_{t}V_{k}(t,\mathbf{x}) &+ \tilde{\mathbf{a}}_{0}(t,\mathbf{x}) \cdot \nabla_{\mathbf{x}}V_{k}(t,\mathbf{x}) \\
&= -\cfrac{1}{\theta}\,\int_{0}^{\theta} (\D_{t}W_{k}+\bm{\alpha}_{0}\cdot\nabla_{\mathbf{x}}W_{k})(t,\sigma,\mathbf{x}) \, d\sigma \\
&\qquad - \sum_{i\,=\,1}^{k}\Bigg[ \cfrac{1}{\theta}\,\int_{0}^{\theta} \bm{\alpha}_{i}(t,\sigma,\mathbf{x}) \cdot \left[ \nabla_{\mathbf{x}}V_{k-i}(t,\mathbf{x})+\nabla_{\mathbf{x}}W_{k-i}(t,\sigma,\mathbf{x})\right] \, d\sigma \Bigg] \, ,
\end{split}
\\
V_{k}(t=0,\mathbf{x}) = 0 \, .
\end{array}
\right.
\end{equation}
\end{theorem} 

\begin{theorem} \label{cor_eq_Uk}
$U_{k}$ satisfies the following equation:
\begin{equation*}
\begin{split}
\D_{t}U_{k}(t,\tau,\mathbf{x}) + \mathbf{a}_{0}(t,\tau,\mathbf{x}) \cdot & \nabla_{\mathbf{x}}U_{k}(t,\tau,\mathbf{x}) = R_{k}(t,\tau,\mathbf{x}) - \sum_{i\,=\,1}^{k}\mathbf{a}_{i}(t,\tau,\mathbf{x})\cdot\nabla_{\mathbf{x}}U_{k-i}(t,\tau,\mathbf{x}) \, ,
\end{split}
\end{equation*}
where $R_{k}$ is obtained from the definition \eqref{def_Ri} extended to the case $i=k$.
\end{theorem}

We can remark that the statements of Theorems \ref{CV_Uk} and \ref{cor_eq_Uk} can be viewed as improvements of the contents of \cite{Two-scale_expansion}. Indeed, assuming that $\mathbf{A}_{\epsilon}$ does not depend on $\epsilon$ and that $\mathbf{L}$ only depends on $t$ and $\mathbf{x}$ leads to 
\begin{equation*}
\mathbf{A}_{\epsilon,i}(t,\mathbf{x}) = \left\{
\begin{array}{ll}
\mathbf{A}(t,\mathbf{x}) \, , & \textnormal{if $i = 0$,} \\
0 \, , & \textnormal{else,}
\end{array}
\right.
\end{equation*}
so $\bm{\alpha}_{i} = \mathbf{a}_{i} = \tilde{\mathbf{a}}_{i} = 0$ for any $i > 0$. Consequently, the expression of $R_{i}$ and $W_{i}$ is reduced to
\begin{equation*}
\begin{split}
R_{i}(t,\tau,\mathbf{x}) 
&= \D_{t}U_{i}(t,\tau,\mathbf{x}) + \mathbf{a}_{0}(t,\tau,\mathbf{x}) \cdot \nabla_{\mathbf{x}}U_{i}(t,\tau,\mathbf{x}) \, ,
\end{split}
\end{equation*}
\begin{equation*}
W_{i}(t,\tau,\mathbf{x}) = \int_{0}^{\tau} \left(\D_{t}U_{i-1}+\mathbf{A}\cdot\nabla_{\mathbf{x}}U_{i-1}\right)\left(t,\sigma,\mathbf{X}(\sigma;\mathbf{x},t;0)\right) \, d\sigma \, .
\end{equation*}
Finally, the transport equation for $V_{i}$ is reduced to
\begin{equation*}
\D_{t}V_{i}(t,\mathbf{x}) + \tilde{\mathbf{a}}_{0}(t,\mathbf{x}) \cdot \nabla_{\mathbf{x}}V_{i}(t,\mathbf{x}) = -\cfrac{1}{\theta}\,\int_{0}^{\theta} (\D_{t}W_{i}+\bm{\alpha}_{0}\cdot\nabla_{\mathbf{x}}W_{i})(t,\sigma,\mathbf{x}) \, d\sigma \, ,
\end{equation*}
for any $i \geq 0$.

\subsection{Application to the Guiding-Center regime}

We now apply the results above to the case of the following linear Vlasov equation:
\begin{equation*}
\left\{
\begin{array}{l}
\displaystyle \D_{t}f_{\epsilon}+\mathbf{v}\cdot\nabla_{\mathbf{x}}f_{\epsilon} + \left(\mathbf{E}_{\epsilon} + \mathbf{v} \times \mathbf{B}_{\epsilon} + \cfrac{\mathbf{v} \times \bm{\beta}_{\epsilon}}{\epsilon} \right) \cdot \nabla_{\mathbf{v}}f_{\epsilon} = 0 \, , \\
f_{\epsilon}(t=0,\mathbf{x},\mathbf{v}) = f^{0}(\mathbf{x},\mathbf{v}) \, .
\end{array}
\right.
\end{equation*}
In this equation, $t \in [0,T]$ is the dimensionless time variable, $\mathbf{x} \in \R^{3}$ is the dimensionless space variable, $\mathbf{v} \in \R^{3}$ is the dimensionless velocity variable, $f_{\epsilon} = f_{\epsilon}(t,\mathbf{x},\mathbf{v}) \in \R$ is the unknown distribution function, $\mathbf{E}_{\epsilon} = \mathbf{E}_{\epsilon}(t,\mathbf{x}) \in \R^{3}$ and $\mathbf{B}_{\epsilon} = \mathbf{B}_{\epsilon}(t,\mathbf{x}) \in \R^{3}$ are the given electric and magnetic fields, $f^{0} = f^{0}(\mathbf{x},\mathbf{v}) \in \R$ is the given initial distribution. We finally assume from now that the vector function $\bm{\beta}_{\epsilon}$ is of the form
\begin{equation*}
\bm{\beta}_{\epsilon}(t,\mathbf{x}) = \bm{\beta}\left(t,\cfrac{t}{\epsilon},\mathbf{x}\right) \, ,
\end{equation*}
where $\bm{\beta} : [0,T] \times \R \times \R^{3} \to \R^{3}$ is a given function assumed to be $\theta$-periodic and continuous in $\tau$ with $\theta > 0$ fixed. \\
\indent Before going further, we introduce additional objects linked to $\bm{\beta}$. First, let $\tilde{\bm{\beta}}$ be defined such that $\D_{\tau}\tilde{\bm{\beta}} = \bm{\beta}$ and $\tilde{\bm{\beta}}(t,0,\mathbf{x}) = 0$ for any $t,\mathbf{x}$. Second, we define the matrix $\tilde{\mathfrak{B}} = \tilde{\mathfrak{B}}(t,\tau,\mathbf{x})$ such that $\tilde{\mathfrak{B}}(t,\tau,\mathbf{x})\, \mathbf{v} = \mathbf{v} \times \tilde{\bm{\beta}}(t,\tau,\mathbf{x})$ for any $t,\tau,\mathbf{x},\mathbf{v}$. Finally, we define $\mathcal{R} = \mathcal{R}(t,\tau,\mathbf{x}) = \exp\left(\tilde{\mathfrak{B}}(t,\tau,\mathbf{v})\right)$. \\
\indent We fix $q > 3/2$ and we consider the following hypotheses: 
\begin{hypothesis}\label{hyp_beta_GC}
We assume that the function $\mathcal{R}$ satisfies 
\begin{itemize}
\item $\mathcal{R}$ is $\theta$-periodic in $\tau$ direction,
\item $\mathcal{R} \in \left(L^{\infty}\left(0,T;L_{\#}^{\infty}\left(0,\theta;W^{1,\infty}(K)\right)\right)\right)^{3\times 3}$ for any compact subset $K \subset \R^{3}$,
\item $\D_{t}\mathcal{R} \in \left(L^{\infty}\left(0,T;L_{\#}^{\infty}\left(0,\theta;W^{1,q}(K)\right)\right)\right)^{3\times 3}$ for any compact subset $K \subset \R^{3}$.
\end{itemize}
\end{hypothesis}

Consequently, adding sufficient hypotheses on $f^{0}$, $(\mathbf{E}_{\epsilon})_{\epsilon\,>\,0}$ and $(\mathbf{B}_{\epsilon})_{\epsilon\,>\,0}$ allows to establish a 0-th order two-scale convergence result: 
\begin{theorem} \label{TSCV_F0_GC}
We assume that Hypotheses \ref{hyp_beta_GC} are satisfied and that $f^{0} \in L^{2}(\R^{6})$ and both sequences $(\mathbf{E}_{\epsilon})_{\epsilon\,>\,0}$ and $(\mathbf{B}_{\epsilon})_{\epsilon\,>\,0}$ are bounded in $\left(L^{\infty}\left(0,T;W^{1,q}(K)\right)\right)^{3}$ independently of $\epsilon$ and for any $K \subset \R^{3}$ compact. We denote $\bm{\mathcal{E}}_{0} = \bm{\mathcal{E}}_{0}(t,\tau,\mathbf{x})$ and $\bm{\mathcal{B}}_{0} = \bm{\mathcal{B}}_{0}(t,\tau,\mathbf{x})$ as the respective two-scale limit of $(\mathbf{E}_{\epsilon})_{\epsilon\,>\,0}$ and $(\mathbf{B}_{\epsilon})_{\epsilon\,>\,0}$ in the space  $\left(L^{\infty}\left(0,T;L_{\#}^{\infty}\left(0,\theta;W^{1,q}(K)\right)\right)\right)^{3}$ and we define $\bm{\mathcal{L}}_{0}$ as
\begin{equation*}
\bm{\mathcal{L}}_{0}(t,\tau,\mathbf{x},\mathbf{v}) = \bm{\mathcal{E}}_{0}(t,\tau,\mathbf{x}) + \mathbf{v} \times \bm{\mathcal{B}}_{0}(t,\tau,\mathbf{x}) \, .
\end{equation*}
Then, $(f_{\epsilon})_{\epsilon\,>\,0}$ is bounded in $L^{\infty}\left(0,T;L^{2}(\R^{6})\right)$ independently of $\epsilon$ and, up to the extraction of a subsequence, two-scale converges to the profile $F_{0} = F_{0}(t,\tau,\mathbf{x},\mathbf{v})$ in $L^{\infty}\left(0,T;L_{\#}^{\infty}\left(0,\theta;L^{2}(\R^{6})\right)\right)$. Furthermore, $F_{0}$ is characterized by
\begin{equation}
F_{0}(t,\tau,\mathbf{x},\mathbf{v}) = G_{0}\left(t,\mathbf{x},\mathcal{R}(t,-\tau,\mathbf{x})\,\mathbf{v}\right) \, ,
\end{equation}
with $G_{0} = G_{0}(t,\mathbf{x},\mathbf{v}) \in L^{\infty}\left(0,T;L_{loc}^{2}(\R^{6})\right)$ solution of
\begin{equation}\label{def_F0_GC}
\left\{
\begin{array}{l}
\begin{split}
\D_{t}G_{0}(t,\mathbf{x},\mathbf{v}) &+ \left(\mathcal{J}_{1}(t,\mathbf{x})\,\mathbf{v}\right)\cdot\nabla_{\mathbf{x}}G_{0}(t,\mathbf{x},\mathbf{v}) + \mathcal{J}_{2}(\bm{\mathcal{L}}_{0})(t,\mathbf{x},\mathbf{v}) \cdot \nabla_{\mathbf{v}}G_{0}(t,\mathbf{x},\mathbf{v}) = 0 \, , 
\end{split}
\\
G_{0}(t=0,\mathbf{x},\mathbf{v}) = f^{0}(\mathbf{x},\mathbf{v}) \, ,
\end{array}
\right.
\end{equation}
where $\mathcal{J}_{1}(t,\mathbf{x})$ and $\mathcal{J}_{2}(\bm{\mathcal{L}}_{0})(t,\mathbf{x},\mathbf{v})$ are defined by
\begin{align*}
\mathcal{J}_{1}(t,\mathbf{x}) &= \cfrac{1}{\theta}\,\int_{0}^{\theta} \mathcal{R}(t,\tau,\mathbf{x}) \, d\tau \, , \\
\mathcal{J}_{2}(\bm{\mathcal{L}}_{0})(t,\mathbf{x},\mathbf{v}) &= \cfrac{1}{\theta}\int_{0}^{\theta} J_{2}(\bm{\mathcal{L}}_{0})(t,\tau,\mathbf{x},\mathbf{v}) \, d\tau \, .
\end{align*}
with $J_{2}$ defined as
\begin{equation*}
\begin{split}
\hspace{-0.5em}J_{2}(\bm{\mathcal{L}}_{0})(t,\tau,\mathbf{x},\mathbf{v}) = \mathcal{R}(t,\tau,\mathbf{x})^{-1} \Big[ -\D_{t}\mathcal{R}(t,\tau,\mathbf{x}) \mathbf{v} &- \big(\nabla_{\mathbf{x}}\mathcal{R}(t,\tau,\mathbf{x})\,\mathbf{v}\big) \mathcal{R}(t,\tau,\mathbf{x})\,\mathbf{v} \\
&+ \bm{\mathcal{L}}_{0}\left(t,\tau,\mathbf{x},\mathcal{R}(t,\tau,\mathbf{x})\mathbf{v}\right) \Big] \, .
\end{split}
\end{equation*}
\end{theorem}

We can remark here that the results of Theorem \ref{TSCV_F0_GC} are coherent with the Guiding-Center model presented in \cite{Homogenization_VP}. Indeed, taking $\mathbf{B}_{\epsilon} = 0$ and $\bm{\beta} = \mathbf{e}_{z}$ leads to the matrix
\begin{equation*}
\mathcal{R}(t,\tau,\mathbf{x}) = \left(
\begin{array}{ccc}
\cos\tau & \sin\tau & 0 \\
-\sin\tau & \cos\tau & 0 \\
0 & 0 & 1
\end{array}
\right) \, ,
\end{equation*}
which is $2\pi$-periodic in $\tau$. Consequently, assuming that $\mathbf{E}_{\epsilon}$ and $\mathbf{B}_{\epsilon}$ converge strongly in $\left(L^{\infty}\left(0,T;L_{loc}^{2}(\R^{3})\right)\right)^{3}$ to $\mathbf{E}$ and $\mathbf{B}$ respectively, we have
\begin{equation*}
\mathcal{J}_{1}(t,\mathbf{x})\mathbf{v} = v_{z}\,\mathbf{e}_{z} \, ,
\end{equation*}
and
\begin{equation*}
\mathcal{J}_{2}(\bm{\mathcal{L}}_{0})(t,\mathbf{x},\mathbf{v}) = E_{z}(t,\mathbf{x})\,\mathbf{e}_{z} + \mathbf{v} \times \left(B_{z}(t,\mathbf{x})\,\mathbf{e}_{z}\right) \, .
\end{equation*}

\indent In order to characterize the higher order terms, it is necessary to add several assumptions. We fix an integer $k > 0$ and we consider the following hypotheses for $(f_{\epsilon})_{\epsilon\,>\,0}$, $(\mathbf{E}_{\epsilon})_{\epsilon\,>\,0}$ and $(\mathbf{B}_{\epsilon})_{\epsilon\,>\,0}$:
\begin{hypothesis} \label{hyp_EiBi_GC}
Defining recursively the sequences $(\mathbf{E}_{\epsilon,i})_{\epsilon\,>\,0}$ and $(\mathbf{B}_{\epsilon,i})_{\epsilon\,>\,0}$ as
\begin{equation*}
\begin{split}
&\left\{
\begin{array}{ll}
\mathbf{E}_{\epsilon,i}(t,\mathbf{x}) = \cfrac{1}{\epsilon}\,\left(\mathbf{E}_{\epsilon,i-1}(t,\mathbf{x}) - \bm{\mathcal{E}}_{i-1}\left(t,\cfrac{t}{\epsilon},\mathbf{x}\right) \right) \, , & \forall\,i=1,\dots,k \, , \\
\mathbf{E}_{\epsilon,0}(t,\mathbf{x}) = \mathbf{E}_{\epsilon}(t,\mathbf{x}) \, ,
\end{array}
\right. \\
&\left\{
\begin{array}{ll}
\mathbf{B}_{\epsilon,i}(t,\mathbf{x}) = \cfrac{1}{\epsilon}\,\left(\mathbf{B}_{\epsilon,i-1}(t,\mathbf{x}) - \bm{\mathcal{B}}_{i-1}\left(t,\cfrac{t}{\epsilon},\mathbf{x}\right) \right) \, , & \forall\,i=1,\dots,k \, , \\
\mathbf{B}_{\epsilon,0}(t,\mathbf{x}) = \mathbf{B}_{\epsilon}(t,\mathbf{x}) \, ,
\end{array}
\right.
\end{split}
\end{equation*}
we assume that, for all $i=0,\dots,k$, $(\mathbf{E}_{\epsilon,i})_{\epsilon\,>\,0}$ and $(\mathbf{B}_{\epsilon,i})_{\epsilon\,>\,0}$ two-scale converge to $\bm{\mathcal{E}}_{i} = \bm{\mathcal{E}}_{i}(t,\tau,\mathbf{x})$ and $\bm{\mathcal{B}}_{i} = \bm{\mathcal{B}}_{i}(t,\tau,\mathbf{x})$ respectively in $\left(L^{\infty}\left(0,T;L_{\#}^{\infty}\left(0,\theta;W^{1,q}(K)\right)\right)\right)^{3}$ for any compact subset $K \subset \R^{3}$.
\end{hypothesis}

\begin{hypothesis} \label{hyp_fi_GC}
Defining recursively the sequence $(f_{\epsilon,i})_{\epsilon\,>\,0}$ as
\begin{equation*}
\left\{
\begin{array}{ll}
f_{\epsilon,i}(t,\mathbf{x},\mathbf{v}) = \cfrac{1}{\epsilon}\,\left(f_{\epsilon,i-1}(t,\mathbf{x},\mathbf{v}) - F_{i-1}\left(t,\cfrac{t}{\epsilon},\mathbf{x},\mathbf{v}\right)\right) \, , & \forall\,i=1,\dots,k-1\, , \\
f_{\epsilon,0}(t,\mathbf{x},\mathbf{v}) = f_{\epsilon}(t,\mathbf{x},\mathbf{v}) \, ,
\end{array}
\right.
\end{equation*}
we assume that, up to a subsequence, the sequence $(f_{\epsilon,i})_{\epsilon\,>\,0}$ two-scale converges to the profile $F_{i} = F_{i}(t,\tau,\mathbf{x},\mathbf{v}) \in L^{\infty}\left(0,T;L_{\#}^{\infty}\left(0,\theta;L^{2}(\R^{6})\right)\right)$ for all $i=0,\dots,k-1$.
\end{hypothesis}

We now define $\bm{\mathcal{L}}_{i}$ for $i=0,\dots,k$ such that
\begin{equation*}
\bm{\mathcal{L}}_{i}(t,\tau,\mathbf{x},\mathbf{v}) = \bm{\mathcal{E}}_{i}(t,\tau,\mathbf{x}) + \mathbf{v} \times \bm{\mathcal{B}}_{i}(t,\tau,\mathbf{x}) \, .
\end{equation*}
We introduce $W_{0},\dots,W_{k}$ such that $W_{0} = 0$ and, for any $i=1,\dots,k$,
\begin{equation} \label{def_Wk_GC}
\begin{split}
W_{i}(t,\tau,\mathbf{x},\mathbf{v}) &= \int_{0}^{\tau} \left( 
\begin{array}{c}
\left(\mathcal{J}_{1}(t,\mathbf{x})-\mathcal{R}(t,\sigma,\mathbf{x})\right)\mathbf{v} \\
\mathcal{J}_{2}(\bm{\mathcal{L}}_{0})(t,\mathbf{x},\mathbf{v}) - J_{2}(\bm{\mathcal{L}}_{0})(t,\sigma,\mathbf{x},\mathbf{v})
\end{array}
\right) \\
&\qquad \qquad \qquad \cdot \left(
\begin{array}{c}
\nabla_{\mathbf{x}}G_{i-1}(t,\mathbf{x},\mathbf{v}) + \nabla_{\mathbf{x}}W_{i-1}(t,\sigma,\mathbf{x},\mathbf{v}) \\
\nabla_{\mathbf{v}}G_{i-1}(t,\mathbf{x},\mathbf{v}) + \nabla_{\mathbf{v}}W_{i-1}(t,\sigma,\mathbf{x},\mathbf{v})
\end{array}
\right) \, d\sigma \\
&\quad + \sum_{j\,=\,1}^{i-1} \int_{0}^{\tau} \Big[ \left[ \mathcal{J}_{3}(\bm{\mathcal{L}}_{j})(t,\mathbf{x},\mathbf{v}) - \mathcal{R}(t,\sigma,\mathbf{x})^{-1}  \bm{\mathcal{L}}_{j}\left(t,\sigma,\mathbf{x},\mathcal{R}(t,\sigma,\mathbf{x})\mathbf{v}\right) \right] \\
&\qquad \qquad \qquad \cdot \left[ \nabla_{\mathbf{v}}G_{i-j-1}(t,\mathbf{x},\mathbf{v}) + \nabla_{\mathbf{v}}W_{i-j-1}(t,\sigma,\mathbf{x},\mathbf{v})\right] \Big] \, d\sigma \\
&\quad - \int_{0}^{\tau} \Bigg[ \D_{t}W_{i-1}(t,\sigma,\mathbf{x},\mathbf{v}) - \cfrac{1}{\theta}\int_{0}^{\theta} \D_{t}W_{i-1}(t,\zeta,\mathbf{x},\mathbf{v})\, d\zeta \Bigg]\, d\sigma \, ,
\end{split}
\end{equation}
with $\mathcal{J}_{3}(\bm{\mathcal{L}}_{j})$ defined by
\begin{equation*}
\mathcal{J}_{3}(\bm{\mathcal{L}}_{j})(t,\mathbf{x},\mathbf{v}) = \cfrac{1}{\theta}\,\int_{0}^{\theta} \mathcal{R}(t,\tau,\mathbf{x})^{-1} \, \bm{\mathcal{L}}_{j}(t,\tau,\mathcal{R}(t,\tau,\mathbf{x})\,\mathbf{v}) \, d\tau \, ,
\end{equation*}
and where $G_{0},\dots,G_{k-1}$ are linked with $F_{0},\dots,F_{k-1}$ and $W_{0},\dots,W_{k-1}$ thanks to the relation
\begin{equation*}
F_{i}(t,\tau,\mathbf{x},\mathbf{v}) = G_{i}\left(t,\mathbf{x},\mathcal{R}(t,-\tau,\mathbf{x})\,\mathbf{v}\right) + W_{i}\left(t,\tau,\mathbf{x},\mathcal{R}(t,-\tau,\mathbf{x})\,\mathbf{v}\right) \, .
\end{equation*}

With these notations, we can establish a two-scale convergence result at the $k$-th order:
\begin{theorem}
We assume that the hypotheses of Theorem \ref{TSCV_F0_GC} and that Hypotheses \ref{hyp_EiBi_GC} and \ref{hyp_fi_GC} are satisfied. We introduce $R_{k-1}$ as follows
\begin{equation}
\begin{split}
R_{k-1}(t,\tau,\mathbf{x},\mathbf{v}) &= \D_{t}F_{k-1}(t,\tau,\mathbf{x},\mathbf{v}) 
+ \left(\mathcal{J}_{1}(t,\mathbf{x})\,\mathbf{v}\right) \cdot \nabla_{\mathbf{x}}F_{k-1}(t,\tau,\mathbf{x},\mathbf{v}) \\
&\quad + \left[\cfrac{1}{\theta}\int_{0}^{\theta}
\mathcal{R}(t,\sigma,\mathbf{x})^{-1}\big[ -\D_{t}\mathcal{R}(t,\sigma,\mathbf{x})\mathbf{v} - (\nabla_{\mathbf{x}}\mathcal{R}(t,\sigma,\mathbf{x})\mathbf{v})\mathcal{R}(t,\sigma,\mathbf{x})\,\mathbf{v} \big] \, d\sigma\right] \\
&\qquad \qquad \qquad \qquad \qquad \qquad \qquad  \cdot \nabla_{\mathbf{v}}F_{k-1}(t,\tau,\mathbf{x},\mathbf{v}) \\
&\quad+ \sum_{i=0}^{k-1} \left[\cfrac{1}{\theta} \int_{0}^{\theta}\mathcal{R}(t,\sigma,\mathbf{x})^{-1}\bm{\mathcal{L}}_{i}\left(t,\sigma+\tau,\mathbf{x},\mathcal{R}(t,\sigma,\mathbf{x})\mathbf{v}\right) \, d\sigma \right] \cdot \nabla_{\mathbf{v}}F_{k-1-i}(t,\tau,\mathbf{x},\mathbf{v}) \, .
\end{split}
\end{equation}
In addition, taking $s'$ such that $\frac{1}{s'} = 1 - \frac{1}{q} - \frac{1}{r}$ with $r \in [1,\frac{6q}{6-q}[$ and defining $X^{s'}(K) = \left(W^{1,q}(K)\right)' \cup \left(W^{1,s'}(K)\right)$, we assume that, for any compact subset $K \subset \R^{6}$,
\begin{itemize}
\item $W_{k} \in L^{\infty}\left(0,T;L_{\#}^{\infty}\left(0,\theta;L^{2}(K)\right)\right)$, 
\item $\D_{t}W_{k},R_{k-1} \in L^{\infty}\left(0,T;L_{\#}^{\infty}\left(0,\theta;X^{s'}(K)\right)\right)$.
\end{itemize}
Then, if the sequence $(f_{\epsilon,k})_{\epsilon\,>\,0}$ defined by
\begin{equation*}
f_{\epsilon,k}(t,\mathbf{x},\mathbf{v}) = \cfrac{1}{\epsilon}\,\left(f_{\epsilon,k-1}(t,\mathbf{x},\mathbf{v}) - F_{k-1}\left(t,\cfrac{t}{\epsilon},\mathbf{x},\mathbf{v}\right)\right)\, ,
\end{equation*}
is bounded independently of $\epsilon$ in $L^{\infty}\left(0,T;L_{loc}^{2}(\R^{6})\right)$, it two-scale converges to the profile $F_{k} = F_{k}(t,\tau,\mathbf{x},\mathbf{v}) \in L^{\infty}\left(0,T;L_{\#}^{\infty}\left(0,\theta;L^{2}(\R^{6})\right)\right)$ up to the extraction of a subsequence. Furthermore, $F_{k}$ is fully characterized by
\begin{equation}
F_{k}(t,\tau,\mathbf{x},\mathbf{v}) = G_{k}\left(t,\mathbf{x},\mathcal{R}(t,-\tau,\mathbf{x})\,\mathbf{v}\right) + W_{k}\left(t,\tau,\mathbf{x},\mathcal{R}(t,-\tau,\mathbf{x})\,\mathbf{v}\right) \, ,
\end{equation}
where $W_{k}$ is defined in \eqref{def_Wk_GC} and where $G_{k} = G_{k}(t,\mathbf{x},\mathbf{v}) \in L^{\infty}\left(0,T;L_{loc}^{2}(\R^{6})\right)$ is the solution of
\begin{equation}
\left\{
\begin{array}{l}
\begin{split}
\D_{t}G_{k}&(t,\mathbf{x},\mathbf{v}) + \left(\mathcal{J}_{1}(t,\mathbf{x})\mathbf{v}\right) \cdot \nabla_{\mathbf{x}}G_{k}(t,\mathbf{x},\mathbf{v}) + \mathcal{J}_{2}(\bm{\mathcal{L}}_{0})(t,\mathbf{x},\mathbf{v}) \cdot \nabla_{\mathbf{v}} G_{k}(t,\mathbf{x},\mathbf{v}) \\
&= -\cfrac{1}{\theta}\int_{0}^{\theta} \left[\D_{t}W_{k}(t,\tau,\mathbf{x},\mathbf{v}) + \left(\mathcal{R}(t,\tau,\mathbf{x})\mathbf{v}\right)\cdot \nabla_{\mathbf{x}}W_{k}(t,\tau,\mathbf{x},\mathbf{v}) \right] \, d\tau \\
&\quad -\cfrac{1}{\theta}\int_{0}^{\theta} J_{2}(\bm{\mathcal{L}}_{0})(t,\tau,\mathbf{x},\mathbf{v}) \cdot \nabla_{\mathbf{v}}W_{k}(t,\tau,\mathbf{x},\mathbf{v})\, d\tau \\
&\quad - \cfrac{1}{\theta}\sum_{i=0}^{k}\int_{0}^{\theta} \left[ \mathcal{R}(t,\tau,\mathbf{x})^{-1}\bm{\mathcal{L}}_{i}\left(t,\tau,\mathbf{x},\mathcal{R}(t,\tau,\mathbf{x})\mathbf{v}\right) \right] \cdot \nabla_{\mathbf{v}}W_{k-i}(t,\tau,\mathbf{x},\mathbf{v}) \, d\tau \\
&\quad - \sum_{i=1}^{k} \mathcal{J}_{3}(\bm{\mathcal{L}}_{i})(t,\mathbf{x},\mathbf{v}) \cdot \nabla_{\mathbf{v}}G_{k-i}(t,\mathbf{x},\mathbf{v}) \, ,
\end{split}
\\
G_{k}(t=0,\mathbf{x},\mathbf{v}) = 0 \, .
\end{array}
\right.
\end{equation}
\end{theorem}

\subsection{Finite Larmor Radius regime}

We focus now on the following linear equation:
\begin{equation*}
\left\{
\begin{array}{l}
\displaystyle \D_{t}f_{\epsilon} + \cfrac{\mathbf{v}_{\perp}}{\epsilon}\cdot\nabla_{\mathbf{x}_{\perp}}f_{\epsilon} + v_{||}\,\D_{x_{||}}f_{\epsilon} + \left(\mathbf{E}_{\epsilon} + \mathbf{v} \times \mathbf{B}_{\epsilon} + \cfrac{\mathbf{v} \times \bm{\mathcal{M}}}{\epsilon} \right) \cdot \nabla_{\mathbf{v}}f_{\epsilon} = 0 \, , \\
f_{\epsilon}(t=0,\mathbf{x},\mathbf{v}) = f^{0}(\mathbf{x},\mathbf{v}) \, ,
\end{array}
\right.
\end{equation*}
in which $(\mathbf{x},\mathbf{v}) \in \R^{3} \times \R^{3}$, $t \in [0,T]$, $f_{\epsilon} = f_{\epsilon}(t,\mathbf{x},\mathbf{v}) \in \R$ is the unknown distribution function, $\mathbf{E}_{\epsilon} = \mathbf{E}_{\epsilon}(t,\mathbf{x}) \in \R^{3}$ is the external electric field, $f^{0} = f^{0}(\mathbf{x},\mathbf{v})$ is the initial distribution function, $\bm{\mathcal{M}} = \mathbf{e}_{z} \in \R^{3}$ and $\mathbf{B}_{\epsilon}=\mathbf{B}_{\epsilon}(t,\mathbf{x}) \in \R^{3}$ constitute the external magnetic field. \\
\indent Thanks to well-chosen hypotheses for $f^{0}$ and the sequences $(\mathbf{E}_{\epsilon})_{\epsilon\,>\,0}$ and $(\mathbf{B}_{\epsilon})_{\epsilon\,>\,0}$, it is possible to establish a 0-th order two-scale convergence result:
\begin{theorem}\label{TSCV0_FLR}
We assume that $f^{0} \in L^{2}(\R^{6})$ and that $(\mathbf{E}_{\epsilon})_{\epsilon\,>\,0}$ and $(\mathbf{B}_{\epsilon})_{\epsilon\,>\,0}$ are bounded independently of $\epsilon$ in $\left(L^{\infty}\left(0,T;W^{1,q}(K)\right) \right)^{3}$ for $q > 3/2$ and for any compact subset $K \subset \R^{3}$. \\
We denote with $\bm{\mathcal{E}}_{0} = \bm{\mathcal{E}}_{0}(t,\tau,\mathbf{x})$ and $\bm{\mathcal{B}}_{0} = \bm{\mathcal{B}}_{0}(t,\tau,\mathbf{x})$ the respective two-scale limits of $(\mathbf{E}_{\epsilon})_{\epsilon\,>\,0}$ and $(\mathbf{B}_{\epsilon})_{\epsilon\,>\,0}$ in $\left(L^{\infty}\left(0,T;L_{\#}^{\infty}\left(0,2\pi;W^{1,q}(K)\right)\right)\right)^{3}$ and we introduce the vector function $\bm{\mathcal{L}}_{0}$ defined by
\begin{equation*}
\bm{\mathcal{L}}_{0}(t,\tau,\mathbf{x},\mathbf{v}) = \bm{\mathcal{E}}_{0}(t,\tau,\mathbf{x}) + \mathbf{v} \times \bm{\mathcal{B}}_{0}(t,\tau,\mathbf{x}) \, .
\end{equation*}
Up to a subsequence, $f_{\epsilon}$ two-scale converges to the profile $F_{0} = F_{0}(t,\tau,\mathbf{x},\mathbf{v})$ in $L^{\infty}\left(0,T;L_{\#}^{\infty}\left(0,2\pi;L^{2}(\R^{6})\right)\right)$ and $F_{0}$ is fully characterized by
\begin{equation}
F_{0}(t,\tau,\mathbf{x},\mathbf{v}) = G_{0}\left(t,\mathbf{x}+\mathcal{R}_{1}(-\tau)\,\mathbf{v}, \mathcal{R}_{2}(-\tau)\,\mathbf{v}\right) \, ,
\end{equation}
where $\mathcal{R}_{1}$, $\mathcal{R}_{2}$ and $G_{0} = G_{0}(t,\mathbf{x},\mathbf{v}) \in L^{\infty}\left(0,T;L_{loc}^{2}(\R^{6})\right)$ satisfy
\begin{equation*} \label{def_matR}
\mathcal{R}_{1}(\tau) = \left(
\begin{array}{ccc}
\sin\tau & 1-\cos\tau & 0 \\
\cos\tau-1 & \sin\tau & 0 \\
0 & 0 & 0
\end{array}
\right) \, ,  \quad \mathcal{R}_{2}(\tau) = \left(
\begin{array}{ccc}
\cos\tau & \sin\tau & 0 \\
-\sin\tau & \cos\tau & 0 \\
0 & 0 & 1
\end{array}
\right) \, ,
\end{equation*}
\begin{equation}\label{FLR_0th}
\left\{
\begin{array}{l}
\begin{split}
\D_{t}G_{0}(t,\mathbf{x},\mathbf{v}) + v_{||}\,\D_{x_{||}}G_{0}(t,\mathbf{x},\mathbf{v}) &+ \mathcal{J}_{1}(\bm{\mathcal{L}}_{0})(t,\mathbf{x},\mathbf{v}) \cdot \nabla_{\mathbf{x}}G_{0}(t,\mathbf{x},\mathbf{v}) \\
&+ \mathcal{J}_{2}(\bm{\mathcal{L}}_{0})(t,\mathbf{x},\mathbf{v}) \cdot \nabla_{\mathbf{v}}G_{0}(t,\mathbf{x},\mathbf{v}) = 0 \, ,
\end{split}
\\
G_{0}(t=0,\mathbf{x},\mathbf{v}) = f^{0}(\mathbf{x},\mathbf{v}) \, ,
\end{array}
\right.
\end{equation}
with $\mathcal{J}_{1}(\bm{\mathcal{L}}_{0})$ and $\mathcal{J}_{2}(\bm{\mathcal{L}}_{0})$ defined by
\begin{equation*}
\mathcal{J}_{i}(\bm{\mathcal{L}}_{0}) = \cfrac{1}{2\pi}\,\int_{0}^{2\pi}\mathcal{R}_{i}(-\tau) \,\bm{\mathcal{L}}_{0}\left(t,\tau,\mathbf{x}+\mathcal{R}_{1}(\tau)\,\mathbf{v}, \mathcal{R}_{2}(\tau)\,\mathbf{v}\right) d\tau \, .
\end{equation*}
\end{theorem}

For obtaining higher order two-scale convergence terms, we first consider a fixed $k \in \N^{*}$ and we assume that the electric and magnetic fields satisfy the following hypotheses:
\begin{hypothesis} \label{hyp_EiBi_FLR}
Defining recursively the sequences $(\mathbf{E}_{\epsilon,i})_{\epsilon\,>\,0}$ and $(\mathbf{B}_{\epsilon,i})_{\epsilon\,>\,0}$ as
\begin{equation*}
\begin{split}
&\left\{
\begin{array}{ll}
\mathbf{E}_{\epsilon,i}(t,\mathbf{x}) = \cfrac{1}{\epsilon}\,\left(\mathbf{E}_{\epsilon,i-1}(t,\mathbf{x}) - \bm{\mathcal{E}}_{i-1}\left(t,\cfrac{t}{\epsilon},\mathbf{x}\right) \right) \, , & \forall\,i=1,\dots,k \, , \\
\mathbf{E}_{\epsilon,0}(t,\mathbf{x}) = \mathbf{E}_{\epsilon}(t,\mathbf{x}) \, ,
\end{array}
\right. \\
&\left\{
\begin{array}{ll}
\mathbf{B}_{\epsilon,i}(t,\mathbf{x}) = \cfrac{1}{\epsilon}\,\left(\mathbf{B}_{\epsilon,i-1}(t,\mathbf{x}) - \bm{\mathcal{B}}_{i-1}\left(t,\cfrac{t}{\epsilon},\mathbf{x}\right) \right) \, , & \forall\,i=1,\dots,k \, , \\
\mathbf{B}_{\epsilon,0}(t,\mathbf{x}) = \mathbf{B}_{\epsilon}(t,\mathbf{x}) \, ,
\end{array}
\right.
\end{split}
\end{equation*}
we assume that, for all $i=0,\dots,k$ and up to the extraction of a subsequence, $(\mathbf{E}_{\epsilon,i})_{\epsilon\,>\,0}$ and $(\mathbf{B}_{\epsilon,i})_{\epsilon\,>\,0}$ two-scale converge to the profiles $\bm{\mathcal{E}}_{i} = \bm{\mathcal{E}}_{i}(t,\tau,\mathbf{x})$ and $\bm{\mathcal{B}}_{i} = \bm{\mathcal{B}}_{i}(t,\tau,\mathbf{x})$ respectively in $\left(L^{\infty}\left(0,T;L_{\#}^{\infty}\left(0,2\pi;W^{1,q}(K)\right)\right)\right)^{3}$ for any compact subset $K \subset \R^{3}$.
\end{hypothesis}

\begin{hypothesis} \label{hyp_fi_FLR}
Defining recursively the sequence $(f_{\epsilon,i})_{\epsilon\,>\,0}$ as
\begin{equation*}
\left\{
\begin{array}{ll}
f_{\epsilon,i}(t,\mathbf{x},\mathbf{v}) = \cfrac{1}{\epsilon}\,\left(f_{\epsilon,i-1}(t,\mathbf{x},\mathbf{v}) - F_{i-1}\left(t,\cfrac{t}{\epsilon},\mathbf{x},\mathbf{v}\right)\right) \, , & \forall\,i=1,\dots,k-1\, , \\
f_{\epsilon,0}(t,\mathbf{x},\mathbf{v}) = f_{\epsilon}(t,\mathbf{x},\mathbf{v}) \, ,
\end{array}
\right.
\end{equation*}
we assume that, up to a subsequence, the sequence $(f_{\epsilon,i})_{\epsilon\,>\,0}$ two-scale converges to the profile $F_{i} = F_{i}(t,\tau,\mathbf{x},\mathbf{v}) \in L^{\infty}\left(0,T;L_{\#}^{\infty}\left(0,2\pi;L^{2}(\R^{6})\right)\right)$ for all $i=0,\dots,k-1$.
\end{hypothesis}

Hence, defining $\bm{\mathcal{L}}_{i}$ as 
\begin{equation*}
\bm{\mathcal{L}}_{i}(t,\tau,\mathbf{x},\mathbf{v}) = \bm{\mathcal{E}}_{i}(t,\tau, \mathbf{x}) + \mathbf{v} \times \bm{\mathcal{B}}_{i}(t,\tau,\mathbf{x}) \, ,
\end{equation*}
for all $i=0,\dots,k$, we define recursively the functions $W_{0},\dots,W_{k}$ such that $W_{0} = 0$ and, for any $i > 0$,
\begin{equation} \label{def_Wi_FLR}
\begin{split}
W_{i}(t,\tau,\mathbf{x},\mathbf{v}) &= \sum_{j\,=\,0}^{i-1}\int_{0}^{\tau} \left(
\begin{array}{c}
\mathcal{J}_{1}(\bm{\mathcal{L}}_{j})(t,\mathbf{x},\mathbf{v})-\mathcal{R}_{1}(-\sigma)\,\bm{\mathcal{L}}_{j}\left(t,\sigma,\mathbf{x}+\mathcal{R}_{1}(\sigma)\,\mathbf{v}, \mathcal{R}_{2}(\sigma)\,\mathbf{v} \right) \\
\mathcal{J}_{2}(\bm{\mathcal{L}}_{j})(t,\mathbf{x},\mathbf{v})-\mathcal{R}_{2}(-\sigma)\,\bm{\mathcal{L}}_{j}\left(t,\sigma,\mathbf{x}+\mathcal{R}_{1}(\sigma)\,\mathbf{v}, \mathcal{R}_{2}(\sigma)\,\mathbf{v} \right)
\end{array}
\right) \\
&\qquad \qquad \qquad \qquad \cdot \left(
\begin{array}{c}
\nabla_{\mathbf{x}}G_{i-1-j}(t,\mathbf{x},\mathbf{v}) + \nabla_{\mathbf{x}}W_{i-1-j}(t,\sigma,\mathbf{x},\mathbf{v}) \\
\nabla_{\mathbf{v}}G_{i-1-j}(t,\mathbf{x},\mathbf{v}) + \nabla_{\mathbf{v}}W_{i-1-j}(t,\sigma,\mathbf{x},\mathbf{v})
\end{array}
\right) \, d\sigma \\
&\quad - \int_{0}^{\tau} \left[ \D_{t}W_{i-1}(t,\sigma,\mathbf{x},\mathbf{v}) - \cfrac{1}{2\pi}\,\int_{0}^{2\pi}\D_{t}W_{i-1}(t,\zeta,\mathbf{x},\mathbf{v}) \, d\zeta \right] \, d\sigma
\end{split}
\end{equation}
where, for $i = 0,\dots,k-1$, $G_{i}$ is defined on $[0,T]\times\R^{6}$ thanks to the relation
\begin{equation*}
\begin{split}
F_{i}(t,\tau,\mathbf{x},\mathbf{v}) &= G_{i}\left(t,\mathbf{x}+\mathcal{R}_{1}(-\tau)\,\mathbf{v}, \mathcal{R}_{2}(-\tau)\,\mathbf{v} \right) + W_{i}\left(t,\tau,\mathbf{x}+\mathcal{R}_{1}(-\tau)\,\mathbf{v}, \mathcal{R}_{2}(-\tau)\,\mathbf{v} \right) \, .
\end{split}
\end{equation*}

Hence we have the following result for obtaining the $k$-th order term $F_{k}$:
\begin{theorem}\label{TSCVk_FLR}
We assume that the hypotheses of Theorem \ref{TSCV0_FLR} and Hypotheses \ref{hyp_EiBi_FLR} and \ref{hyp_fi_FLR} are satisfied for a fixed $k \in \N^{*}$, and we introduce the function $R_{k-1}$ defined by
\begin{equation}
\begin{split}
R_{k-1}(t,\tau,\mathbf{x},\mathbf{v}) &= \D_{t}F_{k-1}(t,\tau,\mathbf{x},\mathbf{v}) + v_{||}\,\D_{x_{||}}F_{k-1}(t,\tau,\mathbf{x},\mathbf{v}) \\
&\quad + \cfrac{1}{2\pi} \sum_{i=0}^{k-1} \Bigg[ \int_{0}^{2\pi} \left(
\begin{array}{c}
\mathcal{R}_{1}(-\sigma)\bm{\mathcal{L}}_{i}\left(t,\sigma+\tau,\mathbf{x}+\mathcal{R}_{1}(\sigma)\mathbf{v}, \mathcal{R}_{2}(\sigma)\mathbf{v}\right) \\
\mathcal{R}_{2}(-\sigma)\bm{\mathcal{L}}_{i}\left(t,\sigma+\tau,\mathbf{x}+\mathcal{R}_{1}(\sigma)\mathbf{v}, \mathcal{R}_{2}(\sigma)\mathbf{v}\right) 
\end{array}
\right) d\sigma \\
&\qquad \qquad \qquad \qquad \qquad \cdot \left(
\begin{array}{c}
\nabla_{\mathbf{x}}F_{k-1-i}(t,\tau,\mathbf{x},\mathbf{v}) \\
\nabla_{\mathbf{v}}F_{k-1-i}(t,\tau,\mathbf{x},\mathbf{v})
\end{array}
\right) \Bigg] \, .
\end{split}
\end{equation}
In addition, taking $s'$ such that $\frac{1}{s'} = 1 - \frac{1}{q} - \frac{1}{r}$ with $r \in [1,\frac{6q}{6-q}[$ and defining $X^{s'}(K) = \left(W^{1,q}(K)\right)' \cup \left(W^{1,s'}(K)\right)$, we assume that, for any compact subset $K \subset \R^{6}$,
\begin{itemize}
\item $W_{k} \in L^{\infty}\left(0,T;L_{\#}^{\infty}\left(0,2\pi;L^{2}(K)\right)\right)$, 
\item $\D_{t}W_{k},R_{k-1} \in L^{\infty}\left(0,T;L_{\#}^{\infty}\left(0,2\pi;X^{s'}(K)\right)\right)$.
\end{itemize}
Then, if the sequence $(f_{\epsilon,k})_{\epsilon\,>\,0}$ defined by
\begin{equation*}
f_{\epsilon,k}(t,\mathbf{x},\mathbf{v}) = \cfrac{1}{\epsilon}\,\left(f_{\epsilon,k-1}(t,\mathbf{x},\mathbf{v}) - F_{k-1}\left(t,\cfrac{t}{\epsilon},\mathbf{x},\mathbf{v}\right)\right)\, ,
\end{equation*}
is bounded independently of $\epsilon$ in $L^{\infty}\left(0,T;L_{loc}^{2}(\R^{6})\right)$, it two-scale converges to the profile $F_{k} = F_{k}(t,\tau,\mathbf{x},\mathbf{v}) \in L^{\infty}\left(0,T;L_{\#}^{\infty}\left(0,2\pi;L^{2}(\R^{6})\right)\right)$ up to the extraction of a subsequence. Furthermore, $F_{k}$ is fully characterized by
\begin{equation}
\begin{split}
F_{k}(t,\tau,\mathbf{x},\mathbf{v}) &= G_{k}\left(t,\mathbf{x}+\mathcal{R}_{1}(-\tau)\,\mathbf{v}, \mathcal{R}_{2}(-\tau)\,\mathbf{v} \right) + W_{k}\left(t,\tau,\mathbf{x}+\mathcal{R}_{1}(-\tau)\,\mathbf{v}, \mathcal{R}_{2}(-\tau)\,\mathbf{v} \right) \, ,
\end{split}
\end{equation}
where $W_{k}$ is defined in \eqref{def_Wi_FLR} and where $G_{k} = G_{k}(t,\mathbf{x},\mathbf{v}) \in L^{\infty}\left(0,T;L_{loc}^{2}(\R^{6})\right)$ is the solution of
\begin{equation}
\left\{
\begin{array}{l}
\begin{split}
\D_{t}&G_{k}(t,\mathbf{x},\mathbf{v}) + v_{||}\,\D_{x_{||}}G_{k}(t,\mathbf{x},\mathbf{v}) \\
&\qquad \qquad + \mathcal{J}_{1}(\bm{\mathcal{L}}_{0})(t,\mathbf{x},\mathbf{v}) \cdot \nabla_{\mathbf{x}}G_{k}(t,\mathbf{x},\mathbf{v}) + \mathcal{J}_{2}(\bm{\mathcal{L}}_{0})(t,\mathbf{x},\mathbf{v}) \cdot \nabla_{\mathbf{v}}G_{k}(t,\mathbf{x},\mathbf{v}) \\
&= -\cfrac{1}{2\pi}\,\int_{0}^{2\pi}\left[ \D_{t}W_{k}(t,\tau,\mathbf{x},\mathbf{v}) + v_{||}\,\D_{x_{||}}W_{k}(t,\tau,\mathbf{x},\mathbf{v}) \right] \, d\tau\\
&\quad - \cfrac{1}{2\pi}\sum_{i\,=\,0}^{k}\int_{0}^{2\pi} \Big[ \left(
\begin{array}{c}
\mathcal{R}_{1}(-\tau) \,\bm{\mathcal{L}}_{i}\left(t,\tau,\mathbf{x}+\mathcal{R}_{1}(\tau)\,\mathbf{v}, \mathcal{R}_{2}(\tau)\,\mathbf{v}\right) \\
\mathcal{R}_{2}(-\tau) \,\bm{\mathcal{L}}_{i}\left(t,\tau,\mathbf{x}+\mathcal{R}_{1}(\tau)\,\mathbf{v}, \mathcal{R}_{2}(\tau)\,\mathbf{v}\right)
\end{array}
\right) \\
&\qquad \qquad \qquad \qquad \cdot \left(
\begin{array}{c}
\nabla_{\mathbf{x}}W_{k-i}(t,\tau,\mathbf{x},\mathbf{v}) \\
\nabla_{\mathbf{v}}W_{k-i}(t,\tau,\mathbf{x},\mathbf{v}) 
\end{array}
\right) \Big] \, d\tau \\
&\quad - \sum_{i\,=\,1}^{k} \left(
\begin{array}{c}
\mathcal{J}_{1}(\bm{\mathcal{L}}_{i})(t,\mathbf{x},\mathbf{v}) \cdot \nabla_{\mathbf{x}}G_{k-i}(t,\mathbf{x},\mathbf{v}) \\
\mathcal{J}_{2}(\bm{\mathcal{L}}_{i})(t,\mathbf{x},\mathbf{v}) \cdot \nabla_{\mathbf{x}}G_{k-i}(t,\mathbf{x},\mathbf{v})
\end{array}
\right) \cdot \left(
\begin{array}{c}
\nabla_{\mathbf{x}}G_{k-i}(t,\mathbf{x},\mathbf{v}) \\
\nabla_{\mathbf{v}}G_{k-i}(t,\mathbf{x},\mathbf{v}) 
\end{array}
\right) \, ,
\end{split}
\\
G_{k}(t=0,\mathbf{x},\mathbf{v}) = 0 \, .
\end{array}
\right.
\end{equation}
\end{theorem}

\subsection{Application to axisymmetric charged particle beams}

In this last example, we focus on the following axisymmetric linear Vlasov equation:
\begin{equation*}
\left\{
\begin{array}{l}
\displaystyle \D_{t}f_{\epsilon}(t,r,v_{r}) + \cfrac{v_{r}}{\epsilon}\,\D_{r}f_{\epsilon}(t,r,v_{r}) + \left(E_{\epsilon}(t,r)-\cfrac{r}{\epsilon}\right)\,\D_{v_{r}}f_{\epsilon}(t,r,v_{r}) = 0 \, , \\
f_{\epsilon}(t=0,r,v_{r}) = f^{0}(r,v_{r}) \, .
\end{array}
\right.
\end{equation*}
In this system, $f_{\epsilon} = f_{\epsilon}(t,r,v_{r})$ is the unknown distribution function of the particles, $E_{\epsilon} = E_{\epsilon}(t,r)$ is the radial component of the external magnetic field, the variables $(t,r,v_{r}) \in [0,T]\times \R \times \R$ stand for the time variable and the radial position and velocity variable, with the convention $f_{\epsilon}(t,r,v_{r}) = f_{\epsilon}(t,-r,-v_{r})$, $E_{\epsilon}(t,r) = -E_{\epsilon}(t,-r)$ (see \cite{Frenod-Gutnic-Hirstoaga,PIC-two-scale,Mouton_2009} for details). \\
\indent The two-scale convergence of $f_{\epsilon}$ at 0-th order has been studied by Fr\'enod, Sonnendr\"ucker and Salvarani in \cite{PIC-two-scale} in a more rich context. We recall here this result:
\begin{theorem}[Fr\'enod, Sonnendr\"ucker, Salvarani \cite{PIC-two-scale}]\label{TSCV_F0_axibeam}
We assume that the initial distribution $f^{0}$ is positive on $\R^{2}$ and that $f^{0} \in L^{1}(\R^{2};rdrdv_{r}) \cap L^{2}(\R^{2};rdrdv_{r})$. We also assume that the sequence $(E_{\epsilon})_{\epsilon\,>\,0}$ is bounded independently of $\epsilon$ in the space $L^{\infty}\left(0,T;W^{1,3/2}(K;rdr)\right)$ for any $K \subset \R$ compact. Then, up to the extraction of a subsequence, $f_{\epsilon}$ two-scale converges to the profile $F_{0} = F_{0}(t,\tau,r,v_{r})$ in $L^{\infty}\left(0,T;L_{\#}^{\infty}\left(0,2\pi;L^{2}(\R^{2};rdrdv_{r})\right)\right)$ and $E_{\epsilon}$ two-scale converges to $\mathcal{E}_{0} = \mathcal{E}_{0}(t,r,v_{r})$ in $L^{\infty}\left(0,T;L_{\#}^{\infty}\left(0,2\pi;W^{1,3/2}(K;rdr)\right)\right)$ for any $K \subset \R$ compact, with $F_{0}$ defined by
\begin{equation}
F_{0}(t,\tau,r,v_{r}) = G_{0}(t,r\cos\tau-v_{r}\sin\tau,r\sin\tau + v_{r}\cos\tau) \, ,
\end{equation}
with $G_{0} = G_{0}(t,r,v_{r}) \in L^{\infty}\left(0,T;L_{loc}^{2}(\R^{2};rdrdv_{r})\right)$ solution of
\begin{equation}
\left\{
\begin{array}{l}
\displaystyle \D_{t}G_{0} + \mathcal{J}_{1}(\mathcal{E}_{0})\,\D_{r}G_{0} + \mathcal{J}_{2}(\mathcal{E}_{0})\,\D_{v_{r}}G_{0} = 0 \, , \\
G_{0}(t=0,r,v_{r}) = f^{0}(r,v_{r}) \, ,
\end{array}
\right.
\end{equation}
where
\begin{align}
\mathcal{J}_{1}(\mathcal{E}_{0})(t,r,v_{r}) &= -\cfrac{1}{2\pi}\,\int_{0}^{2\pi} \sin(\tau) \, \mathcal{E}_{0}(t,\tau,r\cos\tau+v_{r}\sin\tau) \, d\tau \, , \\
\mathcal{J}_{2}(\mathcal{E}_{0})(t,r,v_{r}) &= \cfrac{1}{2\pi}\,\int_{0}^{2\pi} \cos(\tau) \, \mathcal{E}_{0}(t,\tau,r\cos\tau+v_{r}\sin\tau) \, d\tau \, .
\end{align}
\end{theorem}

In order to establish higher order two-scale convergence results, it is necessary to add some hypotheses on the external electric field $E_{\epsilon}$. As in the previous paragraphes, we consider a fixed integer $k > 0$ and we formalize it as follows:
\begin{hypothesis}\label{hyp_Ei_axibeam}
Defining recursively the sequence $(E_{\epsilon,i})_{\epsilon\,>\,0}$ as
\begin{equation*}
\left\{
\begin{array}{ll}
E_{\epsilon,i}(t,r) = \cfrac{1}{\epsilon}\,\left(E_{\epsilon,i-1}(t,r)-\mathcal{E}_{i-1}\left(t,\cfrac{t}{\epsilon},r\right)\right) \, , &\forall\,i=1,\dots,k \, , \\
E_{\epsilon,0}(t,r) = E_{\epsilon}(t,r) \, , &
\end{array}
\right.
\end{equation*}
we assume that, for all $i=0,\dots,k$, $(E_{\epsilon,i})_{\epsilon\,>\,0}$ two-scale converges to the profile $\mathcal{E}_{i} = \mathcal{E}_{i}(t,\tau,r)$ in $L^{\infty}\left(0,T;L_{\#}^{\infty}\left(0,2\pi;W^{1,3/2}(K;rdr)\right)\right)$ for any $K \subset \R$ compact.
\end{hypothesis}

We also add some hypotheses about the two-scale convergence of $f_{\epsilon}$ at $i$-th order for $i=0,\dots,k-1$:
\begin{hypothesis}\label{hyp_fi_axibeam}
Defining recursively the sequence $(f_{\epsilon,i})_{\epsilon\,>\,0}$ as
\begin{equation*}
\left\{
\begin{array}{ll}
f_{\epsilon,i}(t,r,v_{r}) = \cfrac{1}{\epsilon}\,\left(f_{\epsilon,i-1}(t,r,v_{r})-F_{i-1}\left(t,\cfrac{t}{\epsilon},r,v_{r}\right)\right) \, , &\forall\,i=1,\dots,k-1 \, , \\
f_{\epsilon,0}(t,r,v_{r}) = f_{\epsilon}(t,r,v_{r}) \, , &
\end{array}
\right.
\end{equation*}
we assume that, up to the extraction of a subsequence, $(f_{\epsilon,i})_{\epsilon\,>\,0}$ two-scale converges to $F_{i} = F_{i}(t,\tau,r,v_{r})$ in $L^{\infty}\left(0,T;L_{\#}^{\infty}\left(0,2\pi;L^{2}(\R^{2};rdrdvr)\right)\right)$ for $i=0,\dots,k-1$. \\
\end{hypothesis}

Hence we can define recursively $W_{0},\dots,W_{k}$ as follows:
\begin{equation} \label{def_Wi_axibeam}
\begin{split}
W_{i}(t,\tau,r,v_{r}) &= \int_{0}^{\tau} \Bigg[\sum_{j=0}^{i-1} \left(
\begin{array}{c}
\mathcal{J}_{1}(\mathcal{E}_{j})(t,r,v_{r}) + \sin(\sigma)\,\mathcal{E}_{j}(t,\sigma,r\cos\sigma+v_{r}\sin\sigma) \\
\mathcal{J}_{2}(\mathcal{E}_{j})(t,r,v_{r}) - \cos(\sigma)\,\mathcal{E}_{j}(t,\sigma,r\cos\sigma+v_{r}\sin\sigma)
\end{array}
\right) \\
&\qquad \qquad \qquad \qquad \cdot \left(
\begin{array}{c}
\D_{r}G_{i-1-j}(t,r,v_{r}) + \D_{r}W_{i-1-j}(t,\sigma,r,v_{r}) \\
\D_{v_{r}}G_{i-1-j}(t,r,v_{r}) + \D_{v_{r}}W_{i-1-j}(t,\sigma,r,v_{r})
\end{array}
\right) \\
&\qquad \qquad - \D_{t}W_{k-1}(t,\sigma,r,v_{r}) + \cfrac{1}{2\pi}\,\int_{0}^{2\pi} \D_{t}W_{k-1}(t,\zeta,r,v_{r})\,d\zeta \Bigg] \, d\sigma \, , 
\end{split}
\end{equation}
where $G_{0},\dots,G_{k-1}$ are linked to $F_{0},\dots,F_{k-1}$ by the relations
\begin{equation*}
\begin{split}
F_{i}(t,\tau,r,v_{r}) &= G_{i}(t,r\cos\tau-v_{r}\sin\tau,r\sin\tau + v_{r}\cos\tau) + W_{i}(t,\tau,r\cos\tau-v_{r}\sin\tau,r\sin\tau + v_{r}\cos\tau) \, .
\end{split}
\end{equation*}
We finally introduce the function $R_{k-1}$ defined by
\begin{equation}
\begin{split}
R_{k-1}(t,\tau,r,v_{r}) = \D_{t}F_{k-1}(t,\tau,r,v_{r}) + \sum_{j=0}^{k-1} \Bigg[ \cfrac{1}{2\pi}\int_{0}^{2\pi} \left(
\begin{array}{c}
\sin(\sigma)\mathcal{E}_{j}(t,\sigma+\tau,r\cos\sigma+v_{r}\sin\sigma) \\
\cos(\sigma)\mathcal{E}_{j}(t,\sigma+\tau,r\cos\sigma+v_{r}\sin\sigma)
\end{array}
\right) d\sigma \\
\cdot \left(
\begin{array}{c}
\D_{r}F_{k-1-j}(t,\tau,r,v_{r}) \\
\D_{v_{r}}F_{k-1-j}(t,\tau,r,v_{r})
\end{array}
\right) \Bigg] \, .
\end{split}
\end{equation}
Hence we can extend the main result of \cite{Frenod-Gutnic-Hirstoaga} to the $k$-th order:
\begin{theorem}\label{TSCV_Fk_axibeam}
We assume that the hypotheses of Theorem \ref{TSCV_F0_axibeam} and Hypotheses \ref{hyp_Ei_axibeam} and \ref{hyp_fi_axibeam} are satisfied for a fixed $k \in \N^{*}$. In addition, taking $s'$ such that $\frac{1}{s'} = 1 - \frac{1}{q} - \frac{1}{r}$ with $r \in [1,\frac{2q}{2-q}[$ and defining $X^{s'}(K;rdrdv_{r}) = \left(W^{1,q}(K;rdrdv_{r})\right)' \cup \left(W^{1,s'}(K;rdrdv_{r})\right)$, we assume that, for any compact subset $K \subset \R^{2}$,
\begin{itemize}
\item $W_{k} \in L^{\infty}\left(0,T;L_{\#}^{\infty}\left(0,2\pi;L^{2}(K;rdrdv_{r})\right)\right)$, 
\item $\D_{t}W_{k},R_{k-1} \in L^{\infty}\left(0,T;L_{\#}^{\infty}\left(0,2\pi;X^{s'}(K;rdrdv_{r})\right)\right)$.
\end{itemize}
Then, if the sequence $(f_{\epsilon,k})_{\epsilon\,>\,0}$ defined by
\begin{equation*}
f_{\epsilon,k}(t,r,v_{r}) = \cfrac{1}{\epsilon}\,\left(f_{\epsilon,k-1}(t,r,v_{r})-F_{k-1}\left(t,\cfrac{t}{\epsilon},r,v_{r}\right)\right) \, ,
\end{equation*}
is bounded independently of $\epsilon$ in $L^{\infty}\left(0,T;L_{loc}^{2}(\R^{2};rdrdv_{r})\right)$, it two-scale converges to the profile $F_{k}=F_{k}(t,\tau,r,v_{r})$ in $L^{\infty}\left(0,T;L_{\#}^{\infty}\left(0,2\pi;L^{2}(\R^{2};rdrdv_{r})\right)\right)$ with
\begin{equation}
\begin{split}
F_{k}(t,\tau,r,v_{r}) &= G_{k}(t,r\cos\tau-v_{r}\sin\tau,r\sin\tau + v_{r}\cos\tau) \\
&\qquad + W_{k}(t,\tau,r\cos\tau-v_{r}\sin\tau,r\sin\tau + v_{r}\cos\tau) \, ,
\end{split}
\end{equation}
where $W_{k}$ is defined in \eqref{def_Wi_axibeam} and where $G_{k} = G_{k}(t,r,v_{r}) \in L^{\infty}\left(0,T;L_{loc}^{2}(\R^{2};rdrdv_{r})\right)$ is the solution of
\begin{equation}\label{def_Gk_axibeam}
\left\{
\begin{array}{l}
\begin{split}
\D_{t}&G_{k}(t,r,v_{r}) + \mathcal{J}_{1}(\mathcal{E}_{0})(t,r,v_{r})\,\D_{r}G_{k}(t,r,v_{r}) + \mathcal{J}_{2}(\mathcal{E}_{0})(t,r,v_{r})\,\D_{v_{r}}G_{k}(t,r,v_{r}) \\
&= -\cfrac{1}{2\pi}\int_{0}^{2\pi} \D_{t}W_{k}(t,\tau,r,v_{r}) \, d\tau \\
&\quad +\cfrac{1}{2\pi} \sum_{i\,=\,0}^{k}\int_{0}^{2\pi} \sin(\tau)\mathcal{E}_{i}(t,\tau,r\sin\tau+v_{r}\sin\tau)\,\D_{r}W_{k-i}(t,\tau,r,v_{r}) \, d\tau \\
&\quad -\cfrac{1}{2\pi} \sum_{i\,=\,0}^{k}\int_{0}^{2\pi} \cos(\tau)\mathcal{E}_{i}(t,\tau,r\sin\tau+v_{r}\sin\tau)\,\D_{v_{r}}W_{k-i}(t,\tau,r,v_{r}) \, d\tau \\
&\quad - \sum_{i\,=\,1}^{k} \mathcal{J}_{1}(\mathcal{E}_{i})(t,r,v_{r})\,\D_{r}G_{k-i}(t,r,v_{r}) - \sum_{i\,=\,1}^{k} \mathcal{J}_{2}(\mathcal{E}_{i})(t,r,v_{r})\,\D_{v_{r}}G_{k-i}(t,r,v_{r}) \, ,
\end{split}
\\
G_{k}(t=0,r,v_{r}) = 0 \, .
\end{array}
\right.
\end{equation}
\end{theorem}

\section{Characterization of each $U_{k}$}

\indent In this section, we aim to prove the two-scale convergence results presented in Theorems \ref{def_U0}, \ref{eq_U0}, \ref{CV_Uk} and \ref{cor_eq_Uk}. For this purpose, we choose to detail the proofs on the generic equation of the form
\begin{equation}\label{eq_geps_generic}
\hspace{-0.3em}\left\{
\begin{array}{l}
\displaystyle \D_{t}g_{\epsilon}(t,\mathbf{x}) + \mathbf{A}_{\epsilon}(t,\mathbf{x}) \cdot \nabla_{\mathbf{x}}g_{\epsilon}(t,\mathbf{x}) + \cfrac{1}{\epsilon}\,\mathbf{L}\left(t,\cfrac{t}{\epsilon},\mathbf{x}\right) \cdot \nabla_{\mathbf{x}}g_{\epsilon}(t,\mathbf{x}) = \cfrac{1}{\epsilon}\,f_{\epsilon}(t,\mathbf{x}) \, , \\
g_{\epsilon}(t=0,\mathbf{x}) = g^{0}(\mathbf{x}) \, ,
\end{array}
\right.
\end{equation}
in which $f_{\epsilon}$, $\mathbf{A}_{\epsilon}$ and $\mathbf{L}$ are known and where $g_{\epsilon}$ is the unknown. The next lines are structured as follows: first, we detail some two-scale convergence results for the model \eqref{eq_geps_generic} under some well-chosen hypotheses for $\mathbf{A}_{\epsilon}$, $\mathbf{L}$ and $f_{\epsilon}$. Then we apply these results onto the equations satisfied by each $u_{\epsilon,i}$ recursively defined thanks to Hypothesis \ref{def_uepsi}.

\subsection{Two-scale convergence of $g_{\epsilon}$}

We aim to establish some two-scale convergence results for the sequence $(g_{\epsilon})_{\epsilon\,>\,0}$ under some well-chosen hypotheses for $\mathbf{A}_{\epsilon}$, $\mathbf{L}$ and $f_{\epsilon}$. These results are detailed in the following theorem:

\begin{theorem}\label{TSCV_g}
We consider $s' > 0$ such that $\frac{1}{s'} = 1-\frac{1}{q}-\frac{1}{r}$ with $r \in [1,\frac{nq}{n-q}[$ and, for all compact subset $K \subset \R^{n}$, we define $X^{s'}(K) = \left(W_{0}^{1,s'}(K)\right)' \cup \left(W_{0}^{1,q}(K)\right)'$. We assume that $\mathbf{A}_{\epsilon}$ and $\mathbf{L}$ satisfy Hypotheses \ref{hyp_U0} and that $g^{0}$ and $(f_{\epsilon})_{\epsilon\,>\,0}$ have the following properties:
\begin{itemize}
\item $g^{0} \in L^{p}(\R^{n})$, 
\item $f_{\epsilon}$ is bounded independently of $\epsilon$ in $W^{1,\infty}\left(0,T;X^{s'}(K)\right)$ and admits $F = F(t,\tau,\mathbf{x})$ as a two-scale limit in $L^{\infty}\left(0,T;L_{\#}^{\infty}\left(0,\theta;X^{s'}(K)\right)\right)$,
\item $F$ satisfies
\begin{equation} \label{Speriodic}
\forall\, (t,\mathbf{x}) \, , \qquad \int_{0}^{\theta} F(t,\tau,\mathbf{X}\left(\tau;\mathbf{x},t;0)\right)\, d\tau = 0 \, ,
\end{equation}
\item The sequence $(f_{\epsilon,1})_{\epsilon\,>\,0}$ defined by
\begin{equation*}
f_{\epsilon,1}(t,\mathbf{x}) = \frac{1}{\epsilon}\, \left(f_{\epsilon}(t,\mathbf{x})-F\left(t,\cfrac{t}{\epsilon},\mathbf{x}\right)\right) \, ,
\end{equation*}
is bounded independently in $L^{\infty}\left(0,T;X^{s'}(K)\right)$ and two-scale converges to the profile $F_{1} = F_{1}(t,\tau,\mathbf{x})$ in $L^{\infty}\left(0,T;L_{\#}^{\infty}\left(0,\theta;X^{s'}(K)\right)\right)$,
\item Defining the function $S = S(t,\tau,\mathbf{x})$ as
\begin{equation*}
S(t,\tau,\mathbf{x}) = \int_{0}^{\tau} F(t,\sigma,\mathbf{X}\left(\sigma;\mathbf{x},t;0)\right)\, d\sigma \, ,
\end{equation*}
we assume that $S$ lies in $L^{\infty}\left(0,T;L_{\#}^{\infty}\left(0,\theta;L_{loc}^{p}(\R^{n})\right)\right)$ and that $\D_{t}S$ lies in $L^{\infty}\left(0,T;L_{\#}^{\infty}\left(0,\theta;X^{s'}(K)\right)\right)$.
\end{itemize}
If $(g_{\epsilon})_{\epsilon\,>\,0}$ is bounded in $L^{\infty}\left(0,T;L_{loc}^{p}(\R^{n})\right)$, it admits a two-scale limit $G = G(t,\tau,\mathbf{x})$ in the space $L^{\infty}\left(0,T;L_{\#}^{\infty}\left(0,\theta;L^{p}(\R^{n})\right)\right)$ and $G$ is characterized thanks to the relation
\begin{equation}
G(t,\tau,\mathbf{x}) = H\left(t,\mathbf{X}(-\tau;\mathbf{x},t;0)\right) + S\left(t,\tau,\mathbf{X}(-\tau;\mathbf{x},t;0)\right) \, ,
\end{equation}
where $H = H(t,\mathbf{x}) \in L^{\infty}\left(0,T;L_{loc}^{p}(\R^{n})\right)$ satisfies
\begin{equation}\label{eq_H}
\left\{
\begin{array}{l}
\begin{split}
\D_{t}H&(t,\mathbf{x}) + \tilde{\mathbf{a}}_{0}(t,\mathbf{x}) \cdot \nabla_{\mathbf{x}}H(t,\mathbf{x}) \\
&= \cfrac{1}{\theta}\, \int_{0}^{\theta} F_{1}\left(t,\tau,\mathbf{X}(\tau;\mathbf{x},t;0)\right) \, d\tau -\cfrac{1}{\theta}\,\int_{0}^{\theta} \left[\D_{t}S(t,\tau,\mathbf{x}) - \bm{\alpha}_{0}(t,\tau,\mathbf{x})\cdot\nabla_{\mathbf{x}}S(t,\tau,\mathbf{x})\right] d\tau \, ,
\end{split}
\\
H(t=0,\mathbf{x}) = g^{0}(\mathbf{x}) \, . 
\end{array}
\right.
\end{equation}
\end{theorem}

\begin{proof}
Since $(g_{\epsilon})_{\epsilon\,>\,0}$ is assumed to be bounded in $L^{\infty}\left(0,T;L_{loc}^{p}(\R^{n})\right)$ independently of $\epsilon$, it admits a two-scale limit $G = G(t,\tau,\mathbf{x})$ in the functional space $L^{\infty}\left(0,T;L_{\#}^{p}(0,\theta;L^{p}(\R^{n})\right)$. In the same spirit of \cite{Finite_Larmor_radius}, the next step of the proof consists in finding an equation linking the first order derivative of $G$ in $\tau$ to the derivatives of $G$ in $\mathbf{x}$-direction. For this purpose, we consider a test function $\psi = \psi(t,\tau,\mathbf{x})$ defined on $[0,T] \times \R \times \R^{n}$ being $\theta$-periodic in $\tau$ direction and with compact support $K \subset \R^{n}$ in $\mathbf{x}$-direction. We multiply \eqref{eq_geps_generic} by $\psi(t,\frac{t}{\epsilon},\mathbf{x})$ and we integrate the result in $t$ and $\mathbf{x}$. Some integrations by parts give
\begin{equation*}
\begin{split}
\int_{0}^{T} \int_{K} g_{\epsilon}(t,\mathbf{x}) \, \Bigg[ &(\D_{t}\psi)\left(t,\cfrac{t}{\epsilon},\mathbf{x}\right) + \cfrac{1}{\epsilon}\, (\D_{\tau}\psi)\left(t,\cfrac{t}{\epsilon},\mathbf{x}\right) + \mathbf{A}_{\epsilon}(t,\mathbf{x}) \cdot (\nabla_{\mathbf{x}}\psi)\left(t,\cfrac{t}{\epsilon},\mathbf{x}\right) \\
&+ \cfrac{1}{\epsilon}\,\mathbf{L}\left(t,\cfrac{t}{\epsilon},\mathbf{x}\right)\cdot (\nabla_{\mathbf{x}}\psi)\left(t,\cfrac{t}{\epsilon},\mathbf{x}\right) \Bigg] \, d\mathbf{x} \, dt \\
&= -\cfrac{1}{\epsilon}\, \int_{0}^{T} \int_{K} f_{\epsilon}(t,\mathbf{x}) \, \psi\left(t,\cfrac{t}{\epsilon},\mathbf{x}\right) \, d\mathbf{x}\, dt + \int_{K} g^{0}(\mathbf{x})\,\psi(0,0,\mathbf{x}) \, d\mathbf{x} \, .
\end{split}
\end{equation*}
Thanks to the considered assumptions for $g^{0}$, $\mathbf{A}_{\epsilon}$, $\mathbf{L}$ and $f_{\epsilon}$, we can multiply by $\epsilon$ and reach the limit $\epsilon \to 0$. This gives
\begin{equation*}
\begin{split}
\int_{0}^{\theta} \int_{0}^{T} \int_{K} G(t,\tau,\mathbf{x}) \, \Bigg[ \D_{\tau}\psi(t,\tau,\mathbf{x}) + &\mathbf{L}(t,\tau,\mathbf{x})\cdot \nabla_{\mathbf{x}}\psi(t,\tau,\mathbf{x}) \Bigg] \, d\mathbf{x} \, dt \\
&= -\int_{0}^{\theta}\int_{0}^{T} \int_{K} F(t,\tau,\mathbf{x}) \, \psi(t,\tau,\mathbf{x}) \, d\mathbf{x}\, dt \, d\tau \, .
\end{split}
\end{equation*}
This means that $G$ satisfies the following equation in $L^{\infty}\left(0,T;L_{\#}^{\infty}\left(0,\theta;L_{loc}^{p}(\R^{n})\right)\right)$:
\begin{equation*}
\D_{\tau}G(t,\tau,\mathbf{x}) + \mathbf{L}(t,\tau,\mathbf{x}) \cdot \nabla_{\mathbf{x}}G(t,\tau,\mathbf{x}) = F(t,\tau,\mathbf{x}) \, .
\end{equation*}
According to Lemma 2.1 from \cite{Two-scale_expansion} and thanks to the hypothesis \eqref{Speriodic}, we can write $G$ as follows
\begin{equation*}
G(t,\tau,\mathbf{x}) = H\left(t,\mathbf{X}(-\tau;\mathbf{x},t;0)\right) + \int_{0}^{\tau} F\left(t,\sigma,\mathbf{X}(\sigma-\tau;\mathbf{x},t;0)\right) \, d\sigma \, ,
\end{equation*}
with $H = H(t,\mathbf{x}) \in L^{\infty}\left(0,T;L_{loc}^{p}(\R^{n})\right)$. \\

\indent The next step consists in proving that $H$ satisfies \eqref{eq_H}. For this purpose, we introduce the sequence $(h_{\epsilon})_{\epsilon\,>\,0}$ defined as
\begin{equation} \label{link_gh}
g_{\epsilon}(t,\mathbf{x}) = h_{\epsilon}\left(t,\mathbf{X}(-\cfrac{t}{\epsilon};\mathbf{x},t;0)\right) + \int_{0}^{t/\epsilon} F\left(t,\sigma,\mathbf{X}(\sigma-\cfrac{t}{\epsilon};\mathbf{x},t;0)\right) \, d\sigma \, .
\end{equation}
Injecting this relation in \eqref{eq_geps_generic} gives
\begin{equation} \label{eq_heps_eq}
\left\{ \hspace{-0.4em}
\begin{array}{l}
\begin{split}
&\D_{t}h_{\epsilon}(t,\mathbf{x}) + \tilde{\mathbf{A}}_{\epsilon}(t,\mathbf{x}) \cdot \nabla_{\mathbf{x}}h_{\epsilon}(t,\mathbf{x}) \\
&\hspace{0.2em}= f_{\epsilon,1}\left(t,\mathbf{X}\left(\cfrac{t}{\epsilon};\mathbf{x},t;0\right)\right) - (\D_{t}S)\left(t,\cfrac{t}{\epsilon},\mathbf{x}\right) - \tilde{\mathbf{A}}_{\epsilon}(t,\mathbf{x}) \cdot \nabla_{\mathbf{x}}S\left(t,\cfrac{t}{\epsilon},\mathbf{x}\right) ,
\end{split}
\\
h_{\epsilon}(t=0,\mathbf{x}) = g^{0}(\mathbf{x}) \, ,
\end{array}
\right.
\end{equation}
where $\tilde{\mathbf{A}}_{\epsilon}$ is linked to $\mathbf{A}_{\epsilon}$ through the following relation:
\begin{equation*}
\tilde{\mathbf{A}}_{\epsilon}(t,\mathbf{x}) = \left((\nabla_{\mathbf{x}}\mathbf{X})\left(\cfrac{t}{\epsilon};\mathbf{x},t;0\right)\right)^{-1} \left( \mathbf{A}_{\epsilon}\left(t,\mathbf{X}\left(\cfrac{t}{\epsilon};\mathbf{x},t;0\right)\right)-(\D_{t}\mathbf{X})\left(\cfrac{t}{\epsilon};\mathbf{x},t;0\right) \right) \, .
\end{equation*}

From the definition of $h_{\epsilon}$ provided by \eqref{link_gh} and the hypotheses made for $F$ and $(g_{\epsilon})_{\epsilon\,>\,0}$, we can write
\begin{equation*}
\forall\, t, \qquad \left\|h_{\epsilon}(t,\cdot)\right\|_{L^{p}(K)} \leq \left\|g_{\epsilon}(t,\cdot)\right\|_{L^{p}(K)} + \theta\, \left\|F(t,\cdot,\cdot)\right\|_{L_{\#}^{\infty}\left(0,\theta;L^{p}(K)\right)} \, ,
\end{equation*}
for all compact subset $K \subset \R^{n}$ so we deduce that the sequence $(h_{\epsilon})_{\epsilon\,>\,0}$ is bounded independently of $\epsilon$ in $L^{\infty}\left(0,T;L_{loc}^{p}(\R^{n})\right)$ and, up to a subsequence, two-scale converges to $H$ in $L^{\infty}\left(0,T;L_{\#}^{\infty}\left(0,\theta;L^{p}(\R^{n})\right)\right)$. Indeed, if we consider a test function $\psi = \psi(t,\tau,\mathbf{x})$ on $[0,T] \times \R \times \R^{n}$ being $\theta$-periodic in $\tau$ direction and with compact support $K \subset \R^{n}$ in $\mathbf{x}$-direction, we have
\begin{equation*}
\begin{split}
&\lim_{\epsilon \to 0} \int_{0}^{T}\int_{\R^{n}} h_{\epsilon}(t,\mathbf{x}) \, \psi\left(t,\cfrac{t}{\epsilon},\mathbf{x}\right) \, d\mathbf{x}\, dt \\
&\hspace{0.6em}= \lim_{\epsilon \to 0} \int_{0}^{T}\int_{\R^{n}} \left[g_{\epsilon}\left(t,\mathbf{X}\left(\cfrac{t}{\epsilon};\mathbf{x},t;0\right)\right) - S\left(t,\cfrac{t}{\epsilon},\mathbf{x}\right) \right] \psi\left(t,\cfrac{t}{\epsilon},\mathbf{x}\right) \, d\mathbf{x}\, dt \\
&\hspace{0.6em}= \lim_{\epsilon \to 0} \int_{0}^{T}\int_{\R^{n}} \left[ g_{\epsilon}(t,\mathbf{x})\, \psi\left(t,\cfrac{t}{\epsilon},\mathbf{X}\left(-\cfrac{t}{\epsilon};\mathbf{x},t;0\right)\right) - S\left(t,\cfrac{t}{\epsilon},\mathbf{x}\right) \psi\left(t,\cfrac{t}{\epsilon},\mathbf{x}\right)\right] \, d\mathbf{x}\, dt \\
&\hspace{0.6em}= \cfrac{1}{\theta} \, \int_{0}^{\theta} \int_{0}^{T}\int_{\R^{n}} \left[ G(t,\tau,\mathbf{x}) \, \psi\left(t,\tau,\mathbf{X}\left(-\tau;\mathbf{x},t;0\right)\right) - S\left(t,\tau,\mathbf{x}\right) \psi\left(t,\tau,\mathbf{x}\right)\right] \, d\mathbf{x}\, dt \, d\tau \\
&\hspace{0.6em}= \cfrac{1}{\theta} \, \int_{0}^{\theta} \int_{0}^{T}\int_{\R^{n}} \left[ G\left(t,\tau,\mathbf{X}(\tau;\mathbf{x},t;0)\right) - S(t,\tau,\mathbf{x})\right] \psi(t,\tau,\mathbf{x}) \, d\mathbf{x}\, dt \, d\tau \\
&\hspace{0.6em}= \cfrac{1}{\theta} \, \int_{0}^{\theta} \int_{0}^{T}\int_{\R^{n}} H(t,\mathbf{x}) \, \psi(t,\tau,\mathbf{x}) \, d\mathbf{x}\, dt \, d\tau \, .
\end{split}
\end{equation*}
Consequently, $h_{\epsilon}$ weakly-* converges to $H$ in $L^{\infty}\left(0,T;L_{loc}^{p}(\R^{n})\right)$ according to Theorem \ref{TSCV_Allaire}. However, as in \cite{Finite_Larmor_radius}, we are able to obtain a strong convergence result for $h_{\epsilon}$ in a well-chosen functional space:
\begin{lemma} \label{strongCV_heps}
For any compact subset $K \subset \R^{n}$, the sequence $(h_{\epsilon})_{\epsilon\,>\,0}$ strongly converges to $H$ in $L^{\infty}\left(0,T;\left(W_{0}^{1,q}(K)\right)'\right)$.
\end{lemma}

\begin{proof}
The procedure is almost similar to the proof of Lemma 4.1 of \cite{Finite_Larmor_radius}. Indeed, from the assumptions made for the sequences $(g_{\epsilon})_{\epsilon\,>\,0}$, $(\mathbf{A}_{\epsilon})_{\epsilon\,>\,0}$, $\mathbf{L}$, $(f_{\epsilon})_{\epsilon\,>\,0}$ and $(f_{\epsilon,1})_{\epsilon\,>\,0}$, we consider a compact subset $K$ of $\R^{n}$ and we sucessively prove that
\begin{itemize}
\item $(\tilde{\mathbf{A}}_{\epsilon})_{\epsilon\,>\,0}$ is bounded independently of $\epsilon$ in $\left(L^{\infty}\left(0,T;W^{1,q}(K)\right)\right)^{n}$ and satisfies $\nabla_{\mathbf{x}} \cdot \tilde{\mathbf{A}}_{\epsilon} = 0$,
\item $(\tilde{\mathbf{A}}_{\epsilon})_{\epsilon\,>\,0}$ is bounded independently of $\epsilon$ in $\left(L^{\infty}\left(0,T;L^{r}(K)\right)\right)^{n}$ for any $r \in [1,\frac{nq}{n-q}[$,
\item $(\tilde{\mathbf{A}}_{\epsilon}\,h_{\epsilon})_{\epsilon\,>\,0}$ and $(\tilde{\mathbf{A}}_{\epsilon}\,S(\cdot,\frac{\cdot}{\epsilon},\cdot))_{\epsilon\,>\,0}$ are bounded independently of $\epsilon$ in the space  $\left(L^{\infty}\left(0,T;L^{s}(K)\right)\right)^{n}$ with $s$ satisfying $\frac{1}{s} = \frac{1}{q}+\frac{1}{r}$,
\item $\left(\nabla_{\mathbf{x}} \cdot (\tilde{\mathbf{A}}_{\epsilon}\,h_{\epsilon})\right)_{\epsilon\,>\,0}$ and $\left(\nabla_{\mathbf{x}} \cdot (\tilde{\mathbf{A}}_{\epsilon}\,S(\cdot,\frac{\cdot}{\epsilon},\cdot))\right)_{\epsilon\,>\,0}$ are bounded in the space $\left(L^{\infty}\left(0,T;\left(W_{0}^{1,s'}(K)\right)'\right)\right)^{n}$ and consequently in $\left(L^{\infty}\left(0,T;X^{s'}(K)\right)\right)^{n}$ independently of $\epsilon$.
\end{itemize}
In addition of these results, we deduce from the hypotheses on $F$ that $\D_{t}S$ is in $L^{\infty}\left(0,T;L_{\#}^{\infty}\left(0,\theta;L^{p}(K)\right)\right)$. At this point, we distinguish 2 different cases according to the considered value of $s'$ in front of $q$:

\begin{enumerate}
\item Assume that $s' > q$. This leads to the continuous embedding $\left(L^{q}(K)\right)' \subset \left(L^{s'}(K)\right)'$ and, consequently, to the continuous embedding $\left(W_{0}^{1,q}(K)\right)' \subset \left(W_{0}^{1,s'}(K)\right)'$, so  $X^{s'}(K) = \left(W_{0}^{1,s'}(K)\right)'$. On another hand, Rellich's theorem gives the compact embedding $L^{p}(K) \subset \left(W_{0}^{1,q}(K)\right)'$. Hence, $\D_{t}S$ and $f_{\epsilon,1}$ lie in $L^{\infty}\left(0,T;\left(W_{0}^{1,s'}(K)\right)'\right)$, and the sequence $(f_{\epsilon,1})_{\epsilon\,>\,0}$ is bounded independently of $\epsilon$ in this space. Finally, we can write that $(h_{\epsilon})_{\epsilon\,>\,0}$ is bounded independently of $\epsilon$ in the following space:
\begin{equation*}
\mathcal{U} = \left\{ h \in L^{\infty}\left(0,T;L^{p}(K)\right) \, : \, \D_{t}h \in L^{\infty}\left(0,T;\left(W_{0}^{1,s'}(K)\right)'\right) \right\} \, .
\end{equation*}
Aubin-Lions' lemma indicates that $\mathcal{U}$ is compactly embedded in the space $L^{\infty}\left(0,T;\left(W_{0}^{1,q}(K)\right)'\right)$, so $h_{\epsilon}$ weakly-* converges to $H$ in $L^{\infty}\left(0,T;L^{p}(K)\right)$ and strongly converges to $H$ in $L^{\infty}\left(0,T;\left(W_{0}^{1,q}(K)\right)'\right)$.
\item Assume that $s' \leq q$. As a consequence, we have the continuous embedding $\left(W_{0}^{1,s'}(K)\right)' \subset \left(W_{0}^{1,q}(K)\right)'$, $X^{s'}(K) = \left(W_{0}^{1,q}(K)\right)'$ and the compact embedding $L^{p}(K) \subset \left(W_{0}^{1,q}(K)\right)'$ so we are insured that $(\D_{t}h_{\epsilon})_{\epsilon\,>\,0}$ is bounded independently of $\epsilon$ in $L^{\infty}\left(0,T;\left(W_{0}^{1,q}(K)\right)'\right)$ and that $(h_{\epsilon})_{\epsilon\,>\,0}$ is bounded in the functional space $\mathcal{U}$ defined by
\begin{equation*}
\mathcal{U} = \left\{ h \in L^{\infty}\left(0,T;L^{p}(K)\right) \, : \, \D_{t}h \in L^{\infty}\left(0,T;\left(W_{0}^{1,q}(K)\right)'\right) \right\} \, .
\end{equation*}
Applying Aubin-Lions' lemma finally allows us to claim that the weak-* convergence of $h_{\epsilon}$ to $H$ in $L^{\infty}\left(0,T;L^{p}(K)\right)$ is a strong convergence in the space $L^{\infty}\left(0,T;\left(W_{0}^{1,q}(K)\right)'\right)$.
\end{enumerate}
\end{proof}

In order to conclude the proof of Theorem \ref{TSCV_g}, we now consider a test function $\psi = \psi(t,\mathbf{x})$ on $[0,T] \times \R^{n}$ with compact support $K \subset \R^{n}$ in $\mathbf{x}$-direction. If we multiply \eqref{eq_heps_eq} by $\psi(t,\mathbf{x})$, integrate the result in $t$ and $\mathbf{x}$, we obtain
\begin{equation*}
\begin{split}
-\int_{0}^{T}& \int_{\R^{n}} h_{\epsilon}(t,\mathbf{x}) \, \left[ \D_{t}\psi(t,\mathbf{x}) + \tilde{\mathbf{A}}_{\epsilon}(t,\mathbf{x}) \cdot \nabla_{\mathbf{x}}\psi(t,\mathbf{x}) \right] \, d\mathbf{x} \, dt - \int_{\R^{n}} g^{0}(\mathbf{x}) \, \psi(0,\mathbf{x}) \, d\mathbf{x} \\
&= \int_{0}^{T} \int_{\R^{n}} f_{\epsilon,1}(t,\mathbf{x}) \, \psi\left(t,\mathbf{X}\left(-\cfrac{t}{\epsilon};\mathbf{x},t;0\right)\right) \, d\mathbf{x} \, dt \\
&\quad - \int_{0}^{T} \int_{\R^{n}} \left[ (\D_{t}S)\left(t,\cfrac{t}{\epsilon},\mathbf{x}\right)\, \psi(t,\mathbf{x}) - S\left(t,\cfrac{t}{\epsilon},\mathbf{x}\right) \, \tilde{\mathbf{A}}_{\epsilon}(t,\mathbf{x})\cdot\nabla_{\mathbf{x}}\psi(t,\mathbf{x}) \right] \, d\mathbf{x}\, dt \, .
\end{split}
\end{equation*}
Thanks to Lemma \ref{strongCV_heps} and to the hypotheses we have formulated for $\mathbf{A}_{\epsilon}$, $f_{\epsilon,1}$ and $S$, we can write the limit obtained when $\epsilon$ converges to 0:
\begin{equation*}
\begin{split}
\int_{0}^{T} \int_{\R^{n}} &H(t,\mathbf{x}) \left[ \D_{t}\psi(t,\mathbf{x}) + \left[\cfrac{1}{\theta}\int_{0}^{\theta}\bm{\alpha}_{0}(t,\tau,\mathbf{x}) \, d\tau \right] \cdot \nabla_{\mathbf{x}}\psi(t,\mathbf{x}) \right]  d\mathbf{x} \, dt + \int_{\R^{n}} g^{0}(\mathbf{x}) \, \psi(0,\mathbf{x}) \, d\mathbf{x} \\
&= -\cfrac{1}{\theta}\, \int_{0}^{T} \int_{\R^{n}} \int_{0}^{\theta}F_{1}(t,\tau,\mathbf{x}) \, \psi\left(t,\mathbf{X}\left(-\tau;\mathbf{x},t;0\right)\right) \, d\tau\, d\mathbf{x} \, dt \\
&\qquad + \int_{0}^{T} \int_{\R^{n}} \left[\cfrac{1}{\theta}\,\int_{0}^{\theta}\D_{t}S\left(t,\tau,\mathbf{x}\right) \, d\tau\right] \, \psi(t,\mathbf{x}) \, d\mathbf{x}\, dt \\
&\qquad - \int_{0}^{T}\int_{\R^{n}} \left[\cfrac{1}{\theta}\,\int_{0}^{\theta} S\left(t,\tau,\mathbf{x}\right) \, \bm{\alpha}_{0}(t,\tau,\mathbf{x})\, d\tau \right] \cdot\nabla_{\mathbf{x}}\psi(t,\mathbf{x}) \, d\mathbf{x}\, dt \, .
\end{split}
\end{equation*}
This corresponds to the variational formulation of \eqref{eq_H} in $L^{\infty}\left(0,T;L_{loc}^{p}(\R^{n})\right)$.
\end{proof}

\subsection{Identification of each $U_{k}$}

Having Theorem \ref{TSCV_g} in hands, we can apply it for identifying each term $U_{k}$ of the expansion \eqref{expansion}. For obtaining some equations for $U_{0}$, we simply use this theorem with the source term $f_{\epsilon} = 0$ on $[0,T] \times \R^{n}$ for each $\epsilon \geq 0$. As a consequence, assuming that $(u_{\epsilon})_{\epsilon\,>\,0}$ is bounded independently in $\epsilon$ in $L^{\infty}\left(0,T;L_{loc}^{p}(\R^{n})\right)$ in addition of Hypotheses \ref{hyp_U0} is sufficient to get the two-scale convergence of $u_{\epsilon}$ to the profile $U_{0} = U_{0}(t,\tau,\mathbf{x})$ in $L^{\infty}\left(0,T;L_{\#}^{\infty}\left(0,\theta;L^{p}(\R^{n})\right)\right)$ entirely characterized by
\begin{equation*}
U_{0}(t,\tau,\mathbf{x}) = V_{0}\left(t,\mathbf{X}(-\tau;\mathbf{x},t;0)\right) \, ,
\end{equation*}
where $V_{0} = V_{0}(t,\mathbf{x}) \in L^{\infty}\left(0,T;L_{loc}^{p}(\R^{n})\right)$ satisfies
\begin{equation*}
\left\{
\begin{array}{l}
\displaystyle \D_{t}V_{0}(t,\mathbf{x}) + \tilde{\mathbf{a}}_{0}(t,\mathbf{x}) \cdot \nabla_{\mathbf{x}}V_{0}(t,\mathbf{x}) = 0 \, , \\
V_{0}(t=0,\mathbf{x}) = u^{0}(\mathbf{x}) \, .
\end{array}
\right.
\end{equation*}
This is the conclusion of Theorem \ref{def_U0}. For reaching the results of Theorem \ref{eq_U0}, we derive in $\mathbf{x}$ and $t$ the relation \eqref{link_U0V0} and we obtain
\begin{equation*}
\nabla_{\mathbf{x}}U_{0}(t,\tau,\mathbf{x}) = \left((\nabla_{\mathbf{x}}\mathbf{X})(-\tau;\mathbf{x},t;0)\right)^{T}\, (\nabla_{\mathbf{x}}V_{0})\left(t,\mathbf{X}(-\tau;\mathbf{x},t;0)\right) \, ,
\end{equation*}
and
\begin{equation*}
\begin{split}
\D_{t}U_{0}(t,\tau,\mathbf{x}) &= (\D_{t}V_{0})\left(t,\mathbf{X}(-\tau;\mathbf{x},t;0)\right) + \D_{t}\mathbf{X}(-\tau;\mathbf{x},t;0) \cdot (\nabla_{\mathbf{x}}V_{0})\left(t,\mathbf{X}(-\tau;\mathbf{x},t;0)\right) \\
&= \left[\D_{t}\mathbf{X}(-\tau;\mathbf{x},t;0)-\tilde{\mathbf{a}}_{0}\left(t,\mathbf{X}(-\tau;\mathbf{x},t;0)\right) \right] \cdot (\nabla_{\mathbf{x}}V_{0})\left(t,\mathbf{X}(-\tau;\mathbf{x},t;0)\right) \\
&= \left[\left((\nabla_{\mathbf{x}}\mathbf{X})(-\tau;\mathbf{x},t;0)\right)^{-1}\left[\D_{t}\mathbf{X}(-\tau;\mathbf{x},t;0)-\tilde{\mathbf{a}}_{0}\left(t,\mathbf{X}(-\tau;\mathbf{x},t;0)\right) \right] \right]\cdot \nabla_{\mathbf{x}}U_{0}(t,\tau,\mathbf{x}) \, .
\end{split}
\end{equation*}

\indent For identifying the higher order terms, more calculations are needed. First, we consider a fixed integer $k \in \N^{*}$ and we assume that Hypotheses \ref{Hilbert_Aeps} and \ref{Hilbert_ueps_kth} are satisfied at step $k$ and that the results of Theorem \ref{CV_Uk} are true for $i=0,\dots,k-1$, meaning that $U_{0},\dots,U_{k-1}$ are fully characterized. These assumptions authorize the definitions of $\bm{\alpha}_{i}$, $\tilde{\mathbf{a}}_{i}$, $\mathbf{a}_{i}$, $W_{i}$ and $R_{i}$ for any $i=0,\dots,k$ as it is suggested in paragraph \ref{TSCV_Uk}. Then we can write an evolution equation for $u_{\epsilon,i}$ for any $i = 1,\dots,k$ thanks to a recurrence procedure: this equation writes
\begin{equation*}
\left\{
\begin{array}{l}
\begin{split}
\D_{t}u_{\epsilon,i}(t,\mathbf{x}) &+ \mathbf{A}_{\epsilon}(t,\mathbf{x}) \cdot \nabla_{\mathbf{x}}u_{\epsilon,i}(t,\mathbf{x}) + \cfrac{1}{\epsilon}\,\mathbf{L}\left(t,\cfrac{t}{\epsilon},\mathbf{x}\right) \cdot \nabla_{\mathbf{x}}u_{\epsilon,i}(t,\mathbf{x}) \\
&= \displaystyle \cfrac{1}{\epsilon}\,\sum_{j\,=\,0}^{i-1}\left(\mathbf{a}_{j}\left(t,\cfrac{t}{\epsilon},\mathbf{x}\right) - \mathbf{A}_{\epsilon,j}(t,\mathbf{x}) \right) \cdot \nabla_{\mathbf{x}}U_{i-1-j}\left(t,\cfrac{t}{\epsilon},\mathbf{x}\right) - \cfrac{1}{\epsilon}\,R_{i-1}\left(t,\cfrac{t}{\epsilon},\mathbf{x}\right) \, ,
\end{split}
\\
u_{\epsilon,i}(t=0,\mathbf{x}) = 0 \, ,
\end{array}
\right.
\end{equation*}
for any $i > 0$. \\
\indent As a consequence, we aim to apply Theorem \ref{TSCV_g} with $f_{\epsilon}$ defined by
\begin{equation*}
f_{\epsilon}(t,\mathbf{x}) = \sum_{i\,=\,0}^{k-1}\left[ \mathbf{a}_{i}\left(t,\cfrac{t}{\epsilon},\mathbf{x}\right)-\mathbf{A}_{\epsilon,i}(t,\mathbf{x})\right] \cdot \nabla_{\mathbf{x}}U_{k-1-i}\left(t,\cfrac{t}{\epsilon},\mathbf{x}\right) - R_{k-1}\left(t,\cfrac{t}{\epsilon},\mathbf{x}\right) \, .
\end{equation*}
First, we have to verify if the sequence $(f_{\epsilon})_{\epsilon\,>\,0}$ is bounded independently of $\epsilon$ in $L^{\infty}\left(0,T;X^{s'}(K)\right)$ for any compact subset $K \subset \R^{n}$. For this purpose, we first remark that Hypotheses \ref{hyp_U0} and \ref{Hilbert_Aeps} imply that there exists a constant $C = C(K) > 0$ such that
\begin{equation*}
\left\|\mathbf{a}_{i}\left(t,\cfrac{t}{\epsilon},\cdot\right)-\mathbf{A}_{\epsilon,i}(t,\cdot)\right\|_{W^{1,q}(K)} \leq C(K) \, ,
\end{equation*}
for any $t \in [0,T]$ and $\epsilon > 0$. Hence, following the same methodology as in the proof of Lemma \ref{strongCV_heps} and assuming that $R_{k-1}$ is in $L^{\infty}\left(0,T;L_{\#}^{\infty}\left(0,\theta;X^{s'}(K)\right)\right)$ leads to the existence of a constant $C' = C'(K) > 0$ such that
\begin{equation*}
\|f_{\epsilon}(t,\cdot)\|_{X^{s'}(K)} \leq C'(K) \, ,
\end{equation*}
for any $\epsilon > 0$ and $t \in [0,T]$, where the norm $\|\cdot\|_{X^{s'}(K)}$ is either the usual norm on $\left(W_{0}^{1,q}(K)\right)'$ or $\left(W_{0}^{1,s'}(K)\right)'$ according to the sign of $s'-q$. \\
\indent This result indicates that $f_{\epsilon}$ two-scale converges to the profile $F = F(t,\tau,\mathbf{x})$ in $L^{\infty}\left(0,T;L_{\#}^{\infty}\left(0,\theta;X^{s'}(K)\right)\right)$ characterized by
\begin{equation*}
F(t,\tau,\mathbf{x}) = \sum_{i\,=\,0}^{k-1}\left[ \mathbf{a}_{i}(t,\tau,\mathbf{x})-\bm{\mathcal{A}}_{i}(t,\tau,\mathbf{x})\right] \cdot \nabla_{\mathbf{x}}U_{k-1-i}(t,\tau,\mathbf{x}) - R_{k-1}(t,\tau,\mathbf{x}) \, .
\end{equation*}
The next step consists in proving that $W_{k}$ defined by
\begin{equation*}
W_{k}(t,\tau,\mathbf{x}) = S(t,\tau,\mathbf{x}) = \int_{0}^{\tau} F\left(t,\sigma,\mathbf{X}(\sigma;\mathbf{x},t;0)\right) \, d\sigma \, ,
\end{equation*}
is such that
\begin{equation*}
\begin{array}{rcl}
W_{k} & \in & L^{\infty}\left(0,T;L_{\#}^{\infty}\left(0,\theta;L^{p}(K)\right)\right) \, , \\
\D_{t}W_{k} & \in & L^{\infty}\left(0,T;L_{\#}^{\infty}\left(0,\theta; X^{s'}(K)\right)\right) \, , \\
W_{k}(t,\theta,\mathbf{x}) & = & 0 \, , \quad \forall\, t,\mathbf{x} \, .
\end{array}
\end{equation*}
The first two points are handled thanks to the hypotheses for $W_{k}$ which are added for claiming Theorem \ref{CV_Uk}. The last point can be proved by using the definition of $R_{k-1}$. Indeed, we have
\begin{equation} \label{Wktheta}
\begin{split}
W_{k}(t,\theta,\mathbf{x}) &= \int_{0}^{\theta} \left(\sum_{i\,=\,0}^{k-1}\left[ \mathbf{a}_{i}-\bm{\mathcal{A}}_{i}\right] \cdot \nabla_{\mathbf{x}}U_{k-1-i} - R_{k-1}\right)\left(t,\tau,\mathbf{X}(\tau;\mathbf{x},t;0)\right)\, d\tau \, ,
\end{split}
\end{equation}
with
\begin{equation} \label{eq_Rkm1}
\begin{split}
R_{k-1}\left(t,\tau,\mathbf{X}(\tau;\mathbf{x},t;0)\right) &= \D_{t}W_{k-1}(t,\tau,\mathbf{x}) + \tilde{\mathbf{a}}_{0}(t,\mathbf{x}) \cdot \nabla_{\mathbf{x}}W_{k-1}(t,\tau,\mathbf{x}) \\
&\quad - \cfrac{1}{\theta}\,\int_{0}^{\theta}(\D_{t}W_{k-1}+\bm{\alpha}_{0}\cdot\nabla_{\mathbf{x}}W_{k-1})(t,\sigma,\mathbf{x}) \, d\sigma \\
&\quad + \sum_{i\,=\,1}^{k-1}\Bigg[ \cfrac{1}{\theta}\,\int_{0}^{\theta} \bm{\alpha}_{i}(t,\sigma,\mathbf{x}) \cdot \left[ \nabla_{\mathbf{x}}W_{k-1-i}(t,\tau,\mathbf{x}) - \nabla_{\mathbf{x}}W_{k-1-i}(t,\sigma,\mathbf{x}) \right] d\sigma \Bigg] \, ,
\end{split}
\end{equation}
and
\begin{equation} \label{eq_aiAinablaUkmim1}
\begin{split}
&\left(\big[\mathbf{a}_{i} - \bm{\mathcal{A}}_{i} \big] \cdot \nabla_{\mathbf{x}}U_{k-1-i} \right) \left(t,\tau,\mathbf{X}(\tau;\mathbf{x},t;0)\right) \\
&\qquad = \left[(\nabla_{\mathbf{x}}\mathbf{X})(\tau;\mathbf{x},t;0) \left[ \tilde{\mathbf{a}}_{i}(t,\mathbf{x}) - \bm{\alpha}_{i}(t,\tau,\mathbf{x}) \right] \right] \cdot (\nabla_{\mathbf{x}}U_{k-1-i})\left(t,\tau,\mathbf{X}(\tau;\mathbf{x},t;0)\right) \\
&\qquad = \left[ \tilde{\mathbf{a}}_{i}(t,\mathbf{x}) - \bm{\alpha}_{i}(t,\tau,\mathbf{x}) \right] \cdot \left( \nabla_{\mathbf{x}}V_{k-1-i}(t,\mathbf{x}) + \nabla_{\mathbf{x}}W_{k-1-i}(t,\tau,\mathbf{x}) \right) \, ,
\end{split}
\end{equation}
for any $i=0,\dots,k-1$. Hence, using the links between $\tilde{\mathbf{a}}_{i}$, $\bm{\alpha}_{i}$ and $\bm{\mathcal{A}}_{i}$, we inject \eqref{eq_Rkm1} and \eqref{eq_aiAinablaUkmim1} in \eqref{Wktheta} for obtaining a new formulation for $W_{k}$:
\begin{equation*}
\begin{split}
W_{k}(t,\tau,\mathbf{x}) &= \sum_{i\,=\,0}^{k-1} \left[\int_{0}^{\tau} \left(\tilde{\mathbf{a}}_{i}(t,\mathbf{x})-\bm{\alpha}_{i}(t,\sigma,\mathbf{x})\right) \, d\sigma \right] \cdot \nabla_{\mathbf{x}}V_{k-1-i}(t,\mathbf{x}) \\
&\quad + \sum_{i\,=\,0}^{k-1} \int_{0}^{\tau} \Bigg[ \cfrac{1}{\theta}\,\int_{0}^{\theta} \bm{\alpha}_{i}(t,\xi,\mathbf{x})\cdot \nabla_{\mathbf{x}}W_{k-1-i}(t,\xi,\mathbf{x}) \, d\xi - \bm{\alpha}_{i}(t,\sigma,\mathbf{x})\cdot \nabla_{\mathbf{x}}W_{k-1-i}(t,\sigma,\mathbf{x}) \Bigg] \, d\sigma \\
&\quad + \int_{0}^{\tau} \Bigg[\cfrac{1}{\theta}\,\int_{0}^{\theta}\D_{t}W_{k-1}(t,\zeta,\mathbf{x})\, d\zeta - \D_{t}W_{k-1}(t,\sigma,\mathbf{x}) \Bigg]\, d\sigma \, .
\end{split}
\end{equation*}
It becomes straightforward that $W_{k}(t,\theta,\mathbf{x}) = 0$ for any $t$ and for any $\mathbf{x}$. \\
\indent The last property we have have to satisfy for completing the proof of Theorem \ref{CV_Uk} consists in proving that the sequence $(f_{\epsilon,1})_{\epsilon\,>\,0}$ defined by
\begin{equation*}
\begin{split}
f_{\epsilon,1}(t,\mathbf{x}) &= \cfrac{1}{\epsilon}\,\left(f_{\epsilon}(t,\mathbf{x})-F\left(t,\cfrac{t}{\epsilon},\mathbf{x}\right)\right) \\
&= \cfrac{1}{\epsilon}\,\sum_{i\,=\,0}^{k-1}\left[ \bm{\mathcal{A}}_{i}\left(t,\cfrac{t}{\epsilon},\mathbf{x}\right)-\mathbf{A}_{\epsilon,i}(t,\mathbf{x})\right]\,\cdot \nabla_{\mathbf{x}}U_{k-1-i}\left(t,\cfrac{t}{\epsilon},\mathbf{x}\right) \, ,
\end{split}
\end{equation*}
is bounded independently of $\epsilon$ in $L^{\infty}\left(0,T;L_{\#}^{\infty}\left(0,\theta;X^{s'}(K)\right)\right)$ for any compact subset $K \subset \R^{n}$. To obtain this result, we remark that $f_{\epsilon,1}$ can write as
\begin{equation*}
f_{\epsilon,1}(t,\mathbf{x}) = -\sum_{i\,=\,1}^{k}\mathbf{A}_{\epsilon,i}(t,\mathbf{x})\cdot\nabla_{\mathbf{x}}U_{k-i}\left(t,\cfrac{t}{\epsilon},\mathbf{x}\right) \, ,
\end{equation*}
thanks to Hypothesis \ref{Hilbert_Aeps}. Consequently, this sequence admits the profile $F_{1} = F_{1}(t,\tau,\mathbf{x})$ defined as
\begin{equation*}
F_{1}(t,\mathbf{x}) = -\sum_{i\,=\,1}^{k}\bm{\mathcal{A}}_{i}(t,\tau,\mathbf{x})\cdot\nabla_{\mathbf{x}}U_{k-i}(t,\tau,\mathbf{x}) \, ,
\end{equation*}
as a two-scale limit in $L^{\infty}\left(0,T;L_{\#}^{\infty}\left(0,\theta;X^{s'}(K)\right)\right)$. \\
\indent We end the proof of Theorem \ref{CV_Uk} by assuming that $(u_{\epsilon,k})_{\epsilon\,>\,0}$ is bounded independently of $\epsilon$ in $L^{\infty}\left(0,T;L_{loc}^{p}(\R^{n})\right)$: we deduce that $u_{\epsilon,k}$ two-scale converges to the profile $U_{k} = U_{k}(t,\tau,\mathbf{x})$ in $L^{\infty}\left(0,T;L_{\#}^{\infty}\left(0,\theta;L^{p}(\R^{n})\right)\right)$ defined by
\begin{equation*}
U_{k}(t,\tau,\mathbf{x}) = V_{k}\left(t,\mathbf{X}(-\tau;\mathbf{x},t;0)\right) + W_{k}\left(t,\tau,\mathbf{X}(-\tau;\mathbf{x},t;0)\right) \, ,
\end{equation*}
where $V_{k} = V_{k}(t,\mathbf{x}) \in L^{\infty}\left(0,T;L_{loc}^{p}(\R^{n})\right)$ satisfies
\begin{equation*}
\left\{
\begin{array}{l}
\begin{split}
\D_{t}V_{k}(t,\mathbf{x}) + &\tilde{\mathbf{a}}_{0}(t,\mathbf{x}) \cdot \nabla_{\mathbf{x}}V_{k}(t,\mathbf{x}) \\
&= -\cfrac{1}{\theta} \int_{0}^{\theta} \left[ \sum_{i\,=\,1}^{k}(\bm{\mathcal{A}}_{i}\cdot\nabla_{\mathbf{x}}U_{k-i})\left(t,\tau,\mathbf{X}(\tau;\mathbf{x},t;0)\right) \right] \, d\tau \\
&\qquad -\cfrac{1}{\theta} \int_{0}^{\theta} \left[ \D_{t}W_{k}(t,\tau,\mathbf{x}) + \bm{\alpha}_{0}(t,\tau,\mathbf{x})\cdot\nabla_{\mathbf{x}}W_{k}(t,\tau,\mathbf{x}) \right] \, d\tau \, ,
\end{split}
\\
V_{k}(t=0,\mathbf{x}) = \left\{
\begin{array}{ll}
u^{0}(\mathbf{x}) \, , & \textnormal{if $k=0$,} \\ 0 & \textnormal{else.}
\end{array}
\right.
\end{array}
\right.
\end{equation*}

\section{Conclusions and perspectives}

We have proposed some two-scale convergence results for a particular kind of convection equations in which a part of the convection term is singularly perturbed. These results can be viewed as an improvement of the calculations done by Fr\'enod, Raviart and Sonnendr\"ucker in \cite{Two-scale_expansion} since the properties of the convection terms $\mathbf{A}_{\epsilon}$ and $\mathbf{L}$ are less restrictive: indeed, in the present paper, the two-scale convergence can be proved with a $\epsilon$-dependent $\mathbf{A}_{\epsilon}$ and with $\mathbf{L}$ depending on $\epsilon$ in some particular sense. Along with these results, we have described the list of required hypotheses on $(\mathbf{A}_{\epsilon})_{\epsilon\,>\,0}$ and $\mathbf{L}$ for reaching the $k$-th order of two-scale convergence for $(u_{\epsilon})_{\epsilon\,>\,0}$. Finally, we have applied these new results to three different rescaled linear Vlasov equations that canbe considered in the context of MCF or charged particle beams. The limit systems that have been obtained consolidate the existing results and complete them by proposing $k$-th order two-scale limit models. \\
\indent From a numerical point of view, these new informations can be used for enriching the two-scale numerical methods that are currently based on the resolution of the 0-th order limit model: in particular, the limit model presented in Theorem \ref{TSCV_F0_axibeam} is discretized for approaching the solution of \eqref{Vlasov_axibeam_intro} but these numerical experiments are relevant for $\epsilon \ll 1$ (see \cite{PIC-two-scale,Mouton_2009}). Combining this approach with the numerical resolution of higher order two-scale limit models like \eqref{def_Gk_axibeam} may provide some relevant numerical results for values of $\epsilon$ which are less close to 0.

\end{document}